\documentclass[11pt]{amsart}

\usepackage{epigamath}

\usepackage[english]{babel}

\usepackage{mathrsfs}
\usepackage{array, boldline, makecell, booktabs}
\definecolor{gray}{gray}{0.4}
\usepackage{amssymb}
\usepackage[all]{xy}
\usepackage{multirow,multicol}

\newtheorem{theorem}{Theorem}[section]
\newtheorem{pro}[theorem]{Proposition}
\newtheorem{lemma}[theorem]{Lemma}
\newtheorem{cor}[theorem]{Corollary}

\theoremstyle{definition}

\theoremstyle{remark}

\newtheorem{remark}[theorem]{Remark}

\def\lexp#1#2{\kern\scriptspace\vphantom{#2}^{#1}\kern-\scriptspace#2}

\def\equat{\refstepcounter{theorem}\begin{equation}}
\def\endequat{\end{equation}}

\def\vide{\varnothing}

\def\fb{{\mathbf f}}

\def\Ab{{\mathbf A}}
\def\Gb{{\mathbf G}}
\def\Lb{{\mathbf L}}
\def\Pb{{\mathbf P}}
\def\Sb{{\mathbf S}}

\def\Hrm{{\mathrm{H}}}
\def\Nrm{{\mathrm{N}}}
\def\Trm{{\mathrm{T}}}
\def\Wrm{{\mathrm{W}}}
\def\Zrm{{\mathrm{Z}}}

\def\AC{{\mathcal{A}}}
\def\EC{{\mathcal{E}}}
\def\HC{{\mathcal{H}}}
\def\KC{{\mathcal{K}}}
\def\OC{{\mathcal{O}}}
\def\RC{{\mathcal{R}}}
\def\SC{{\mathcal{S}}}
\def\UC{{\mathcal{U}}}
\def\XC{{\mathcal{X}}}
\def\ZC{{\mathcal{Z}}}

\def\SG{{\mathfrak S}}

\def\AM{{\mathbb{A}}}
\def\CM{{\mathbb{C}}}
\def\PM{{\mathbb{P}}}
\def\QM{{\mathbb{Q}}}

\def\gti{{\tilde{g}}}
\def\Gti{{\tilde{G}}}
\def\Xti{{\tilde{X}}}

\def\Gamt{{\tilde{\Gamma}}}

\def\Xhat{{\hat{X}}}

\def\OCh{{\hat{\mathcal{O}}}}

\def\a{\alpha}
\def\g{\gamma}
\def\d{\delta}
\def\r{\rho}
\def\s{\sigma}
\def\t{\tau}
\def\z{\zeta}

\def\G{\Gamma}
\def\O{\Omega}

\def\mub{{\boldsymbol{\mu}}}

\def\bfit{\bfseries\itshape}

\def\DS{\displaystyle}
\def\SSS{\scriptscriptstyle}
\def\slv{{\SSS{\mathrm{SL}}}}

\DeclareMathOperator{\Ker}{{\mathrm{Ker}}}
\DeclareMathOperator{\Refle}{Ref}

\DeclareMathOperator{\Der}{Der}

\DeclareMathOperator{\Spec}{{\mathrm{Spec}}}

\def\codim{\operatorname{codim}\nolimits}

\def\Deg{{\mathrm{Deg}}}
\def\Codeg{{\mathrm{Codeg}}}

\def\longto{\longrightarrow}

\newcommand{\longiso}{\stackrel{\sim}{\longrightarrow}}

\newcommand{\longinjto}{\lhook\joinrel\longrightarrow}

\def\petitespace{\vphantom{\DS{\frac{\DS{A}}{\DS{A}}}}}

\def\fonction#1#2#3#4#5{\begin{array}{rccc}
{#1} : & {#2} & \longto & {#3}  \\
& {#4} & \longmapsto & {#5}
\end{array}}

\def\itemth#1{\item[${\mathrm{(#1)}}$]}

\definecolor{shadecolor}{gray}{0.90}

\def\boitegrise#1#2{\begin{centerline}{\fcolorbox{black}{shadecolor}{~
    \begin{minipage}[t]{#2}{\vphantom{~}#1\vphantom{$A_{\DS{A_A}}$}}
            \end{minipage}~}}\end{centerline}\medskip}

\makeatletter
\def\hlinewd#1{%
\noalign{\ifnum0=`}\fi\hrule \@height #1 %
\futurelet\reserved@a\@xhline} \makeatother

\newlength\epaisLigne
\newcommand\clinewd[2]{\noalign{\global\epaisLigne\arrayrulewidth\global\arrayrulewidth #1}%
\cline{#2} \noalign{\global\arrayrulewidth\epaisLigne}}

\def\sing{{\mathrm{sing}}}
\def\smooth{{\mathrm{sm}}}

\def\xyinj{\ar@{^{(}->}}

\EpigaVolumeYear{5}{2021} \EpigaArticleNr{3} \ReceivedOn{June 17,
2020}
\InFinalFormOn{November 8, 2020} \AcceptedOn{December 23, 2020}

\title{Complex reflection groups and K3 surfaces I}
\titlemark{Reflection groups and K3 surfaces}

\author{C\'edric Bonnaf\'e}
\address{IMAG, Universit\'e de Montpellier, CNRS, Montpellier, France}
\email{cedric.bonnafe@umontpellier.fr}

\author{Alessandra Sarti}
\address{Universit\'e de Poitiers,
Laboratoire de Math\'ematiques et Applications, UMR 7348 du CNRS,
TSA 61125, 11 bd Marie et Pierre Curie, 86073 Poitiers cedex 9,
France} \email{sarti@math.univ-poitiers.fr}

\authormark{C. Bonnaf\'e and A. Sarti}

\AbstractInEnglish{\footnotesize{We construct here many families of
K3 surfaces that one can obtain as quotients of algebraic surfaces
by some subgroups of the rank four complex reflection groups. We
find in total 15 families with at worst $ADE$--singularities. In
particular we classify all the K3 surfaces that can be obtained as
quotients by the derived subgroup of the previous complex reflection
groups. We prove our results by using the geometry of the weighted
projective spaces where these surfaces are embedded and the theory
of Springer and Lehrer-Springer on properties of complex reflection
groups. This construction generalizes a previous construction by W.
Barth and the second author.}}



\TitleInFrench{Groupes de r\'eflexions complexes et surfaces K3 I}

\AbstractInFrench{\footnotesize{Nous construisons ici de nombreuses
familles de surfaces K3 que l'on peut obtenir comme quotients de
surfaces alg\'ebriques par certains sous-groupes des groupes de
r\'eflexions complexes de rang 4. Nous trouvons au total 15 familles
avec au plus des singularit\'es $ADE$. En particulier, nous
classifions toutes les surfaces K3 pouvant \^etre obtenues comme
quotients par le sous-groupe d\'eriv\'e des groupes de r\'eflexions
complexes susmentionn\'es. Nous d\'emontrons nos r\'esultats en
utilisant la g\'eom\'etrie des espaces projectifs \`a poids dans
lesquels ces surfaces sont plong\'ees et la th\'eorie de  Springer
et Lehrer-Springer sur les propri\'et\'es des groupes de
r\'eflexions complexes. Cette construction g\'en\'eralise une
construction pr\'ec\'edente introduite par W.~Barth et la seconde
auteure.}}

\acknowledgement{\footnotesize{The first author is partly supported
by the ANR: Projects No ANR-16-CE40-0010-01 (GeRepMod) and
ANR-18-CE40-0024-02 (CATORE).}}

\dedication{In memory of Gianandrea}

\begin{document}


\removeabove{0.6cm} \removebetween{0.4cm} \removebelow{0.7cm}

\maketitle

\begin{prelims}

\DisplayAbstractInEnglish

\bigskip

\DisplayKeyWords

\medskip

\DisplayMSCclass

\bigskip

\languagesection{Fran\c{c}ais}

\bigskip

\DisplayTitleInFrench

\medskip

\DisplayAbstractInFrench

\end{prelims}


\newpage

\setcounter{tocdepth}{1}

\tableofcontents


\section{Introduction}
In this paper we describe a relation between complex reflection
groups and K3 surfaces. A relation already appeared recently in the
paper \cite{bonnafe-sarti-m20} by the authors, where they use
the reflection group denoted by $G_{29}$ in the Shepard--Todd
classification \cite{shephardtodd} to describe K3 surfaces with
maximal finite automorphism groups containing the Mathieu group
$M_{20}$. The motivation for our paper is an early paper of the
second author \cite{sartipencil} and of W. Barth and the second
author \cite{barth-sarti} where they study first one parameter
families of surfaces of general type in the three dimensional
complex projective space containing four surfaces with a high number
of nodes (i.e. $A_1$--singularities). Then they study the quotients
of these families by some groups related to the platonic solids:
tetrahedron, octahedron and icosahedron and which they call
bipolyhedral groups. These turn out to be subgroups of some complex
reflection groups and they show that the quotients are K3 surfaces
with $ADE$--singularities. In this paper we show that these examples
are only a few examples of K3 surfaces that one can produce by using
complex reflection groups. Moreover the theory of Springer and
Lehrer--Springer  and some technical lemmas allow a deep
understanding of the reason why the quotients have trivial dualizing
sheaf and admit only $ADE$--singularities. This allows then to
conclude that the minimal resolution are K3 surfaces. We find in
total 15 families of K3 surfaces.

More precisely we consider a complex reflection group $W$ acting on
a four dimensional complex vector space $V$. By
Shephard--Todd/Chevalley/Serre Theorem \cite[Theorem~4.1]{broue}
there exist 4 algebraically independent polynomials which are
invariant under the action of $W$ and which generate algebraically
the ring of all $W$-invariant polynomials. We assume furthermore
that $W$ is generated by reflections of order $2$ and in Table
\ref{table:degres} we give the list of the degrees of the four
invariant polynomials (observe that these degrees do not depend on
the polynomials) and of the codegrees corresponding to the degrees
of four invariant derivatives which generate the module of all
$W$-invariant derivatives. The aim of the paper is to study the
quotient of the projective zero set $\ZC(f)$ of an homogeneous
fundamental invariant $f$ of $W$ by some subgroup $\G$ of $W$: the
derived subgroup $W'$ and the group $W^\slv=\Ker(\det) \cap W$. A
reason for this choice is that the simple structure of the invariant
ring of $W$ and the fact that $W$ is generated by reflections of
order $2$ imply that $\ZC(f)/\G$ is a complete intersection in a
weighted projective space. More precisely the quotient surface
$\ZC(f)/W^\slv$ is a double cover of a weighted projective plane
whereas if $W'$ is different from $W^\slv$ then $\ZC(f)/W'$ is a
complete intersection in a four dimensional weighted projective
space. We also explain how to obtain explicit equations for
$\ZC(f)/W'$ or $\ZC(f)/W^\slv$.

It turns out that if the degree $d$ of $f$ is ''well chosen'' and $\ZC(f)$ 
has only ADE-singularities, then
$\ZC(f)/\G$ is a K3 surface with $ADE$--singularities (``well
chosen'' means that the sum of the degrees of the equations of the
surface is equal to the sum of the weights of the ambient weighted
projective space: the set of pairs $(W,d)$ such that $d$ is
``well-chosen'' is denoted by $\KC_3$ and is described in
~\S\ref{sub:k3}).
 To show that we have $ADE$--singularities on the double covers one has to study carefully the singularities of the branching curve as well as the singularities that one gets from the singularities of the weighted projective plane.
Our main Theorem \ref{theo:main} is the following:

\begin{theorem}
Assume that $(W,d) \in \KC_3$. Let $\G$ be the subgroup $W'$ or
$W^\slv$ of $W$ and let $f$ be a fundamental invariant of $W$ of
degree $d$ whose projective zero set $\ZC(f)$ has only
$ADE$--singularities.

Then $\ZC(f)/\G$ is a K3 surface with $ADE$--singularities.
\end{theorem}

In particular the theorem allows us to classify all the K3 surfaces
that can be obtained as quotient by $W'$ or $W^\slv$.
Theorem~\ref{theo:main} is a qualitative result, that insures that
one can build from invariants of some complex reflection groups of
rank $4$ several families of K3 surfaces with $ADE$--singularities.
However, it does not say anything about the types of the
singularities and important invariants of their minimal resolution
(rank of the Picard number, description of the transcendental
lattice).

Theorem~\ref{theo:main} (and its proof) improves previous works by
Barth and the second author~\cite{barth-sarti} for two reasons:
\begin{itemize}
\item[-] By looking at all complex reflection groups of rank $4$,
it enlarges considerably the class of examples of K3 surfaces that
can be constructed as above. It shows moreover that the discovery of
families of K3 surfaces in~\cite{barth-sarti} is not just ''an
accident'' but it is strictly related to polynomial invariants of
complex reflection groups and their action on these.

\item[-] The main difficulty is to prove that $\ZC(f)/\G$ has only ADE singularities.
Our proof involves some general facts about singularities (see
Appendix~\ref{app:auto-ade}) and more complex reflection group
theory (as the theory of Springer and Lehrer--Springer on
eigenspaces of elements of complex reflection groups): as a
consequence, our proof avoids as much as possible (but not
completely) a case-by-case analysis, and so may also be viewed not
only as a generalization but also as an enlightenment of results
from~\cite{barth-sarti}.
\end{itemize}
An important feature of the K3 surfaces constructed in
Theorem~\ref{theo:main} is that most of them have big Picard number,
and generally as big as possible compare to the number of moduli of
the family they belong to. In particular, we can build in this way
several (more than thirty) K3 surfaces with Picard number $20$,
often called {\it singular} K3 surfaces, whose moduli space is a
$0$-dimensional subspace of the $20$-dimensional moduli space of K3
surfaces. This will be explained in the sequel to this
paper~\cite{bonnafe-sarti-2},~\cite{bonnafe-sarti-3}, where we aim
to complete the qualitative results of this first part by
quantitative results whenever $W$ is assumed to be {\it primitive}
(i.e. $W=G_{28}$, $G_{29}$, $G_{30}$ or $G_{31}$). We will for
instance compute the transcendental lattice of the singular K3
surfaces, and describe explicit elliptic fibrations in most cases.
Note that the case where $(W,d)=(G_{28},6)$ or $(G_{30},12)$ and
$\G=W'$ was already treated by Barth and the second
author~\cite{barth-sarti}: these examples will be revisited thanks
to our new techniques, and more geometrical informations will be
given. Note also that, by taking Galois invariant models for complex
reflection groups as in~\cite{marin-michel}, it turns out that all
our families of K3 surfaces are defined over $\QM$ and, in
particular, all the singular K3 surfaces we obtain are defined over
$\QM$: this fact is only checked
case-by-case~\cite{bonnafe-sarti-2},~\cite{bonnafe-sarti-3} but
would deserve a general explanation.

The paper is organized as follows: Section \ref{sec:notation}
contains basic facts on the action of groups of matrices on
homogeneous polynomials and Section~\ref{refl} recalls facts on
reflection groups, in particular we find equations for the quotient
surfaces and we recall basic facts from Springer and Lehrer-Springer
theory, that we use in the next sections (and in the sequel to this
paper~\cite{bonnafe-sarti-2},~\cite{bonnafe-sarti-3}) to describe
the singularities that we have on the quotient surfaces. In Section
\ref{singu} we give several useful facts to describe the
singularities of the quotient surfaces in particular in the case
that these are a double cover of a weighted projective plane. Note
that Sections~\ref{sec:notation}-\ref{singu} are written in a
greater generality (reflection groups acting on vector spaces of any
finite dimension) as they might be of general interest. In Section
\ref{invK3} we describe how to obtain K3 surfaces. In Table
\ref{table:eq-k3} we recall the degrees of the equations and the
weighted projective spaces where the (singular) surfaces are
embedded.  We give in this section the main part of the proof of our
main Theorem \ref{theo:main}. We finish the proof in Section
\ref{sec:singularites} where we show that the quotient K3 surfaces
have at worst $ADE$--singularities.

\bigskip

\noindent{\bf Acknowledgements.} We wish to thank Sylvain Brochard
for useful discussions about the results of
Appendix~\ref{app:auto-ade}. We thank also Enrica Floris and
\'Etienne Mann for several interesting discussions.

\section{Notation, preliminaries}\label{sec:notation}

If $d \ge 1$, we denote by $\mub_d$ the group of $d$-th roots of
unity in $\CM^\times$ and we fix a primitive $d$-th root of unity
$\z_d$ (we will use the standard notation $i=\z_4$). If $l_1,\dots,
l_r$ are positive integers, then $\PM(l_1,\dots,l_r)$ denotes the
corresponding weighted projective space.

We fix an $n$-dimensional $\CM$-vector space $V$ and we denote by
$\PM(V)$ its associated projective space. If $v \in V \setminus
\{0\}$, we denote by $[v] \in \PM(V)$ the line it defines (i.e.
$[v]=\CM v$). The group $\Gb\Lb_\CM(V)$ acts on the algebra $\CM[V]$
of polynomial functions on $V$ as follows: if $g \in \Gb\Lb_\CM(V)$
and $f \in \CM[V]$, we write
$$\lexp{g}{f}(v)=f(g^{-1} \cdot v).$$
If $g \in \Gb\Lb_\CM(V)$ and $\z \in \CM^\times$, we denote by
$V(g,\z)$ the $\z$-eigenspace of $g$. If $v \in V(g,\z)$ and $f \in
\CM[V]^g$, then \equat\label{eq:derive} \lexp{g}{(\partial_v f)} =
\z \partial_v f.
\endequat

If $G$ is a subgroup of $\Gb\Lb_\CM(V)$, we write $PG$ for its image
in $\Pb\Gb\Lb_\CM(V)$. Recall that a subgroup $G$ of $\Gb\Lb_\CM(V)$
is called {\it primitive} if there does not exist a decomposition
$V=V_1 \oplus \cdots \oplus V_r$ with $V_i \neq 0$ and $r \ge 2$
such that $G$ permutes the $V_i$'s. If $S$ is a subset of $V$, we
denote by $G_S$ (resp. $G(S)$) the setwise (resp. pointwise)
stabilizer of $S$ (so that $G(S)$ is a normal subgroup of $G_S$ and
$G_S/G(S)$ acts faithfully on $S$). Note that $G_S=G_{\CM S}$ and
$G(S)=G(\CM S)$, where $\CM S$ denotes the linear span of $S$. The
derived subgroup of $G$ will be denoted by $G'$, and we set
$G^\slv=G \cap \Sb\Lb_\CM(V)$. Note that $G' \subset G^\slv$ and
that the inclusion might be strict. We state here a totally trivial
result which will be used extensively and freely along this series
of papers:

\begin{lemma}\label{lem:trivial}
Let $g \in \Gb\Lb_\CM(V)$, let $\z$ be a root of unity of order $d$,
let $v \in V$ be such that $g(v)=\z v$ and let $f \in \CM[V]^g$ be
homogeneous of degree $e$ not divisible by $d$. Then $f(v)=0$.
\end{lemma}

\begin{proof}
As $f \in \CM[V]^g$, we have $f(g(v))=f(v)$. But $f(g(v))=f(\z
v)=\z^{e} f(v)$ because $f$ is homogeneous of degree $e$. So the
result follows from the fact that $\z^e \neq 1$.
\end{proof}

If $X$ is a complex algebraic variety and $x \in X$, we denote by
$\Trm_x(X)$ the tangent space of $X$ at $x$. If $f \in \CM[V]$ is
homogeneous, we will denote by $\ZC(f)$ the projective (possibly
non-reduced) hypersurface in $\PM(V) \simeq \PM^{n-1}$ defined by
$f$. Its singular locus will be denoted by $\ZC_\sing(f)$.

\begin{lemma}\label{lem:trivial-tangent}
Let $G$ be a finite subgroup of $\Gb\Lb_\CM(V)$, let $f \in
\CM[V]^G$ be homogeneous and let $v \in V^G \setminus\{0\}$ be such
that $f(v)=0$. Then $G$ acts trivially on
$\Trm_{[v]}(\PM(V))/\Trm_{[v]}(\ZC(f))$.
\end{lemma}

\begin{proof}
Since $G$ is finite, there exists a $G$-stable subspace $E$ of $V$
such that $V=E \oplus \CM v$.
Let $\a \in V^*$ be such that $\a(v)=1$ and $\a(E)=0$. The affine
chart $U_\a$ of $\PM(V)$ defined by $\a \neq 0$ is identified with
$v+E$ and, after translation, is identified with $E$: through this
identification, $\ZC(f) \cap U_\a$ is the affine hypersurface
defined by the polynomial $F \in \CM[E]$, where $F(e)= f(v+e)$.
Since $v$ is $G$-invariant, $F$ is also $G$-invariant. Let us denote
by $F_i$ its $i$-th homogeneous component: it is $G$-invariant. Then
$F_0=0$ because $f(v)=0$. Moreover, $\Trm_{[v]}(\PM(V)) = E$ and
$\Trm_{[v]}(\ZC(f))=\Ker(F_1)$ (and these identifications are
$G$-equivariant since $v \in V^G$).

But $\CM F_1$ is the dual space to $E/\Ker(F_1)$: since $G$ acts
trivially on $\CM F_1$, this shows that $G$ acts trivially on
$E/\Ker(F_1) = \Trm_{[v]}(\PM(V))/\Trm_{[v]}(\ZC(f))$.
\end{proof}

The next result is just an easy generalization
of~\cite[\S{6}]{sartipencil}.

\begin{cor}\label{coro:v-fixe-singulier}
Let $G$ be a finite subgroup of $\Gb\Lb_\CM(V)$ such that $\dim V^G
= 1$ and let $f \in \CM[V]^G$ be non-zero and homogeneous. We assume
that $f$ vanishes at $V^G$, viewed as a point of $\PM(V)$. Then
$V^G$ is a singular point of $\ZC(f)$.
\end{cor}

\begin{proof}
Let $v \in V^G \setminus \{0\}$. We keep the notation of the proof
of the previous Lemma~\ref{lem:trivial-tangent} ($E$, $\a$, $F$,
$F_i$). Since $V^G=\CM v$, we have $E^G=0$ and so, by
semisimplicity, we have that $(E/\Ker(F_1))^G=0$. But $G$ acts
trivially on $E/\Ker(F_1)$ by Lemma~\ref{lem:trivial-tangent}.
Therefore, $E/\Ker(F_1)=0$ so $\Trm_{[v]}(\ZC(f))=\Ker(F_1)=E$ and
so $\ZC(f)$ is singular at $[v]=V^G$.
\end{proof}

\begin{remark}\label{rem:maximaux-singuliers}
The previous lemma might be used to explain the construction of
several singular curves and surfaces constructed by the two
authors~\cite{sartipencil},~\cite{bonnafesingular}. Let us explain
how to proceed.

Let $G$ be a finite subgroup of $\Gb\Lb_\CM(V)$, and let
$H_1,\dots, H_r$ be a set of representatives of conjugacy classes
of maximal subgroups of $G$ among the subgroups $H$ satisfying
$\dim(V^H)=1$. Let $N_k=N_G(H_k)$, let $v_k \in V^{H_k} \setminus
\{0\}$ and let $\O_k$ denote the $G$-orbit of $[v_k]$ in $\PM(V)$.
Then \equat\label{eq:g-orbite} |\O_k|=\frac{|G|}{|N_k|}.
\endequat
For this, it is sufficient to prove that $N_k=G_{[v_k]}$. First,
$N_k$ stabilizes $V^{H_k}=[v_k]$, which proves that $N_k \subset
G_{[v_k]}$. Conversely, $G_{[v_k]}$ normalizes $G_{v_k}$. But $H_k
\subset G_{v_k}$ by construction and, by the maximality of $H_k$,
this implies that $H_k=G_{v_k}$.

Now, we fix two linearly independent homogeneous polynomials $f_1$
and $f_2$ of the same degree such that $f_1(v_k) \neq 0$ for all
$k$. We also set $\lambda_k=f_2(v_k)/f_1(v_k)$. Then it follows from
Corollary~\ref{coro:v-fixe-singulier} that
\equat\label{eq:singulier-maximal} \text{\it $\ZC(f_2-\lambda_k
f_1)$ contains $\O_k$ in its singular locus.}
\endequat
It also shows that, if $G$ is defined over a subfield $K$ of $\CM$,
then the points of $\O_k$ (which are singular points of
$\ZC(f_2-\lambda_k f_1)$) have coordinates in $K$.
\end{remark}

\begin{cor}\label{coro:transverse}
Let $G$ be a finite subgroup of $\Gb\Lb_\CM(V)$ such that $\dim V^G
= 2$, let $f \in \CM[V]^G$ be homogeneous and non-zero and let $v
\in V^G \setminus\{0\}$ be such that $f(v)=0$. Let $L$ be the line
$\PM(V^G)$ and assume that $[v]$ is a smooth point of $\ZC(f)$. Then
the intersection of $L$ with $\ZC(f)$ is transverse at $[v]$.
\end{cor}

\begin{proof}
We keep again the notation of the proof of
Lemma~\ref{lem:trivial-tangent} ($E$, $\a$, $F$, $F_i$). Since $\dim
V^G=2$, this forces $\dim E^G = 1$. Since $[v]$ is smooth, this
means that $F_1 \neq 0$. It then follows from
Lemma~\ref{lem:trivial-tangent} that $E=E^G \oplus \Ker(F_1)$. But
$E^G=\Trm_{[v]}(L)$ and $\Ker(F_1)=\Trm_{[v]}(\ZC(f))$. This shows
the result.
%
%
%
\end{proof}

\begin{cor}\label{coro:droite contenue}
Let $G$ be a finite subgroup of $\Gb\Lb_\CM(V)$ such that $\dim V^G
= 2$, and let $f \in \CM[V]^G$ be homogeneous and non-zero. Let $L$
be the line $\PM(V^G)$ and assume that $L \subset \ZC(f)$. Then $L
\subset \ZC_\sing(f)$.
\end{cor}

\begin{proof}
Let $v \in V^G \setminus \{0\}$ and assume that $[v]$ is a smooth
point of $\ZC(f)$. Then the intersection of $L$ with $\ZC(f)$ is not
transverse at $[v]$ because $L \subset \ZC(f)$: this contradicts
Corollary~\ref{coro:transverse}.
\end{proof}

\section{Reflection groups}\label{refl}

We fix a finite subgroup $W$ of $\Gb\Lb_\CM(V)$ and we set
$$\Refle(W)=\{s \in W~|~\dim(V^s)=n-1\}.$$

\smallskip

\boitegrise{{\bf Hypothesis.} {\it We assume throughout this paper
that
$$W=\langle \Refle(W) \rangle.$$
In other words, $W$ is a {\bfit complex reflection group}.
The number $\codim(V^W)$ is called the {\bfit rank} of
$W$.}}{0.8\textwidth}

\smallskip


Standard arguments allow to reduce most questions about reflection
groups to questions about {\it irreducible} reflection groups. These
last ones have been classified by Shephard-Todd and we refer to
Shephard-Todd numbering~\cite{shephardtodd} for such groups: there
is an infinite family $G(de,e,r)$ with $d$, $e$, $r \ge 0$ (they are
of rank $r$ if $(d,e) \neq (1,1)$ and of rank $r-1$ otherwise) and
$34$ exceptional ones numbered from $G_4$ to $G_{37}$ (they are
exactly the primitive complex reflection groups). If $W$ can be
realized over the field of real numbers, then it is a Coxeter group
and we will also use the notation $\Wrm(X_i)$ where $X_i$ is the
type of some Coxeter graph. For instance, the group $G_{30}$ in
Shephard-Todd numbering is the Coxeter group $\Wrm(H_4)$.

\subsection{Invariants}
By Shephard-Todd/Chevalley/Serre Theorem~\cite[Theorem~4.1]{broue},
there exist $n$ algebraically independent homogeneous elements
$f_1$, $f_2,\dots, f_n$ of $\CM[V]^W$ such that
$$\CM[V]^W=\CM[f_1,f_2,\dots,f_n].$$
Let $d_i=\deg(f_i)$. 
A family $(f_1,f_2,\dots,f_n)$ satisfying the above property is
called a {\it family of fundamental invariants} of $W$. Observe that
by a result of Marin-Michel~\cite{marin-michel}, these polynomials
can be defined over the rational numbers (for more details,
see~\cite{bonnafe-sarti-2}). A homogeneous element $f \in \CM[V]^W$
is called a {\it fundamental invariant} if it belongs to a family of
fundamental invariants. Whereas such a family is not uniquely
defined even up to permutation, the list $(d_1,d_2,\dots,d_n)$ is
well-defined up to permutation and is called the list of {\it
degrees} of $W$: it will be denoted by $\Deg(W)$.

\bigskip

\boitegrise{{\bf Notation.} {\it From now on, and until the end of
this paper, we fix a family $\fb=(f_1,f_2,\dots,f_n)$ of fundamental
invariants and we set $d_i=\deg(f_i)$.}}{0.8\textwidth}

\bigskip

The following equalities are well-known~\cite[Theorem~4.1]{broue}:
\equat\label{eq:degres} |W|=d_1d_2\cdots d_n \qquad\text{and}\qquad
|\Refle(W)|=\sum_{i=1}^n(d_i-1).
\endequat
Also, as $W$ acts irreducibly on $V$, its center $|\Zrm(W)|$
consists of homotheties, so it is cyclic.
Moreover by \cite[Proposition~4.6]{broue}, \equat\label{eq:centre}
|\Zrm(W)|={\mathrm{Gcd}}(d_1,d_2,\dots,d_n).
\endequat
The $\CM[V]$-module $\Der(\CM[V])$ of derivatives of the algebra
$\CM[V]$ is naturally graded in such a way that $\partial_v$ has
degree $-1$ for all $v \in V$. By Solomon
Theorem~\cite[Theorem~4.44~and~\S{4.5.4}]{broue}, the graded
$\CM[V]^W$-module $\Der(\CM[V])^W$ of invariant derivatives is free
of rank $n$, hence it admits a homogeneous $\CM[V]^W$-basis
$(D_1,\dots,D_n)$ whose respective degrees are denoted by
$d_1^*\dots, d_n^*$. Again, the family $(D_1,\dots,D_n)$ is not
uniquely defined even up to permutation, but the list
$(d_1^*,d_2^*,\dots,d_n^*)$ is well-defined up to permutation and is
called the list of {\it codegrees} of $W$: it will be denoted by
$\Codeg(W)$.

We conclude this subsection by a general easy result which follows
immediately from the fact that $\CM[V]^W$ is a graded polynomial
algebra whose weights are given by $\Deg(W)$.

\begin{pro}\label{prop:projectif-a-poids}
The map
$$\fonction{\pi_\fb}{\PM(V)}{\PM(d_1,d_2,\dots,d_n)}{[v]}{[f_1(v) :
f_2(v) : \cdots : f_n(v)]}$$ is well-defined and induces an
isomorphism
$$\PM(V)/W \longiso \PM(d_1,d_2,\dots,d_n).$$
Moreover, $\pi_\fb$ induces by restriction an isomorphism
$$\ZC(f_1)/W \longiso \PM(d_2,\dots,d_n).$$
\end{pro}

\subsection{Reflecting hyperplanes}\label{sub:reflecting}
If $s \in \Refle(W)$, then the hyperplane $V^s$ is called the
reflecting hyperplane of $s$ (or a reflecting hyperplane of $W$). We
denote by $\AC$ the set of {\it reflecting hyperplanes} of $W$. If
$X$ is a subset of $V$, then, by Steinberg-Serre
Theorem~\cite[Theorem~4.7]{broue}, $W(X)$ is generated by
reflections and so is generated by the reflections whose reflecting
hyperplane contains $X$: such a subgroup is called a {\it parabolic
subgroup} of $W$.

If $H \in \AC$, then the group $W(H)$ is cyclic (indeed, by
semisimplicity, it acts faithfully on $V/H$ which has dimension $1$)
and we denote its order by $e_H$. Note that $W_H \setminus \{1\}$ is
the set of reflections of $W$ whose reflecting hyperplane is $H$, so
\equat\label{eq:ref-hyp} |\Refle(W)|=\sum_{H \in \AC} (e_H-1).
\endequat
We denote by $\a_H$ an element of $V^*$ such that $H=\Ker(\a_H)$. In
particular, if all the reflections have order $2$, then
$|\Refle(W)|=|\AC|$.
Finally, note the following equality~\cite[Remark~4.48]{broue}
\equat\label{eq:codeg} |\AC|=\sum_{i=1}^n (d_i^*+1).
\endequat
If $\O$ is a $W$-orbit in $\AC$, then we denote by $e_\O$ the common
value of the $e_H$'s for $H \in \O$. We then set
$$J_\O = \prod_{H \in \O} \a_H.$$
Then there exists a unique polynomial $P_{\fb,\O}$ in variables
$x_1,\dots, x_n$ of respective weights $d_1,\dots, d_n$ such
that \equat\label{eq:j-omega} J_\O^{e_\O} =
P_{\fb,\O}(f_1,\dots,f_n).
\endequat
Note that $P_{\fb,\O}$ is homogeneous of degree $e_\O |\O|$. Then
(see~\cite[Theorem~2.3~and~Corollary~4.3]{stanley}
or~\cite[Theorem~9.19~and~Corollary~9.21]{lehrertaylor})
\equat\label{eq:inv-derive} \CM[V]^{W'}=\CM[f_1,\dots,f_n,(J_\O)_{\O
\in \AC/W}]
\endequat
and a presentation of $\CM[V]^{W'}$ is given by the
relations~(\ref{eq:j-omega}). Consequently:

\begin{pro}\label{prop:zf-dw}
Let $\O_1,\dots, \O_r$ denote the $W$-orbits in $\AC$. Then the
map
$$\fonction{\pi_\fb'}{\PM(V)}{
\PM(d_1,d_2,\dots,d_n,|\O_1|,\dots,|\O_r|)}{[v]}{ [f_1(v) : f_2(v) :
\cdots : f_n(v) : J_{\O_1}(v) : \cdots : J_{\O_r}(v)]}$$ is
well-defined and induces an isomorphism
$$\PM(V)/W' \longiso \{[x_1:\cdots:x_n:j_1:\cdots:j_r] \in
\PM(d_1,\dots,d_n,|\O_1|,\dots,|\O_r|)~|~\hphantom{AAA}$$
$$\hphantom{AAAAAA}
\forall~1\le k \le r,~ j_k^{e_{\O_k}} =
P_{\fb,\O_k}(x_1,\dots,x_n)\}.$$ Moreover, $\pi_\fb'$ induces by
restriction an isomorphism
$$\ZC(f_1)/W' \longiso
\{[x_2:\cdots:x_n:j_1:\dots:j_r] \in
\PM(d_2,\dots,d_n,|\O_1|,\dots,|\O_r|)~|~\hphantom{AAA}$$
$$\hphantom{AAAAAA}
\forall~1 \le k \le r,~j_k^{e_{\O_k}} =
P_{\fb,\O_k}(0,x_2,\dots,x_n)\}.$$
\end{pro}

%




Note that $W/W^\slv$ is cyclic but $W^\slv$ is not necessarily equal
to the derived subgroup $W'$ of $W$. We have $W'=W^\slv$ if and only
if $|\AC/W|=1$. Now, let
$$J=\prod_{H \in \AC} \a_H=\prod_{\O \in \AC/W} J_\O  \in \CM[V].$$
It is well-defined up to a scalar and homogeneous of degree $|\AC|$.
It is the generator of the ideal of the reduced subscheme of the
ramification locus of the morphism $V \to V/W$.
Then by \cite[Remark~3.10~and~Proposition~4.4]{broue}
\equat\label{eq:jacobien} \lexp{w}{J}=\det(w)^{-1} J
\endequat
for all $w \in W$. In particular $J \in \CM[V]^{W^\slv}$ and
$J^{|W/W^\slv|} \in \CM[V]^W$. So there exists a unique polynomial
$P_\fb \in \CM[X_1,\dots,X_n]$, which is homogeneous of degree
$|W/W^\slv|\cdot |\AC|$ if we assign to $X_i$ the degree $d_i$, and
such that \equat\label{eq:j} J^{|W/W^\slv|}=P_\fb(f_1,\dots,f_n).
\endequat

\begin{pro}\label{prop:zf-wsl}
Assume that the map $H \mapsto e_H$ is constant on $\AC$ $($and let
$e$ denote this constant value, which coincides with $|W/W^\slv|)$.
Then $\CM[V]^{W^\slv}=\CM[f_1,f_2,\dots,f_n,J]$ and a presentation
is given by the single equation~\emph{(\ref{eq:j})}.

So the map
$$\fonction{\pi_\fb^\slv}{\PM(V)}{\PM(d_1,d_2,\dots,d_n,|\AC|)}{[v]}{[f_1(v) :
f_2(v) : \cdots : f_n(v) : J(v)]}$$ is well-defined and induces an
isomorphism
$$\PM(V)/W^\slv \longiso \{[x_1:\cdots:x_n:j] \in \PM(d_1,\dots,d_n,|\AC|)~|~
j^e=P_\fb(x_1,\dots,x_n)\}.$$ Moreover, $\pi_\fb^\slv$ induces by
restriction an isomorphism
$$\ZC(f_1)/W^\slv \longiso \{[x_2:\cdots:x_n:j] \in
\PM(d_2,\dots,d_n,|\AC|)~|~ j^e=P_\fb(0,x_2,\dots,x_n)\}.$$
\end{pro}

%
%
%

\subsection{Eigenspaces, Springer theory}
We now recall the basics of Springer and Lehrer-Springer theory: all
the results stated in this subsection can be found
in~\cite{springer},~\cite{lehrerspringer1},~\cite{lehrerspringer2}.
Note that some of the proofs have been simplified
in~\cite{lehrermichel}. Let us fix now a natural number $e$. We set
$$\D(e)=\{1 \le k \le n~|~e~\text{divides}~d_k\},$$
$$\D^*(e)=\{1 \le k \le n~|~e~\text{divides}~d_k^*\},$$
$$\d(e)=|\D(e)|\qquad\text{and}\qquad \d^*(e)=|\D^*(e)|.$$
With this notation, we have \equat\label{eq:max-dim} \d(e)=\max_{w
\in W} \bigl(\dim V(w,\z_e)\bigr).
\endequat
In particular, $\z_e$ is an eigenvalue of some element of $W$ if and
only if $\d(e) \neq 0$ that is, if and only if $e$ divides some
degree of $W$. In this case, we fix an element $w_e$ of $W$ such
that
$$\dim V(w_e,\z_e)=\d(e).$$
We set for simplification $V(e)=V(w_e,\z_e)$ and
$W(e)=W_{V(e)}/W(V(e))$: this subquotient of $W$ acts faithfully on
$V(e)$.

If $f \in \CM[V]$, we denote by $f^{[e]}$ its restriction to $V(e)$.
Note that if $i \not\in \D(e)$, then $f_i^{[e]}=0$ by
Lemma~\ref{lem:trivial}. Let us recall here the results of Springer and Lehrer-Springer 
we will need:

\begin{theorem}[Springer, Lehrer-Springer]\label{theo:springer}
Assume that $\d(e) \neq 0$. Then:
\begin{itemize}
\itemth{a} If $w \in W$, then there exists $x \in W$ such that
$x (V(w,\z_e)) \subset V(e)$.

\itemth{b} $W(e)$ acts (faithfully) on $V(e)$
as a group generated by reflections.

\itemth{c} The family
$(f_k^{[e]})_{k \in \D(e)}$ is a family of fundamental invaraints of
$W(e)$. In particular, the list of degrees of $W(e)$ consists of the
degrees of $W$ which are divisible by $e$ (i.e. $\Deg(W(e))=(d_k)_{k
\in \D(e)}$).

\itemth{d} We have
$$\bigcup_{w \in W} V(w,\z_e)=\bigcup_{x \in W} x(V(e)) =
\{v \in V~|~\forall~k \in \{1,2,\dots,n\} \setminus
\D(e),~f_k(v)=0\}.$$

\itemth{e} $\d^*(e) \ge \d(e)$ with equality if and only if
$W(V(w_e,\z_e)) = 1$.

\itemth{f} If $\d^*(e)=\d(e)$, then
$W(e)=W_{V(e)}=C_W(w_e)$ and the family of eigenvalues (with
multiplicity) of $w_e$ is equal to $(\z_e^{1-d_k})_{1 \le k \le n}$.
Moreover, if $w$ is such that $\dim V(w,\z_e)=\d(e)$, then $w$ is
conjugate to $w_e$.
\end{itemize}
\end{theorem}

\begin{remark}\label{rem:stab-cyclique}
Let $k \in \{1,2,\dots,n\}$ be such that $\d(d_k)=\d^*(d_k)=1$. Then
$z_k=V(d_k)$ is a line in $V$, so we can view it as an element of
$\PM(V)$. By Theorem~\ref{theo:springer}(f), the stabilizer
$W_{z_k}$ of $z_k$ acts faithfully on $V(d_k)$, so it is cyclic and
contains $w_{d_k}$. In fact,
$$W_{z_k}=\langle w_{d_k} \rangle.$$
For proving this, let $e=|W_{z_k}|$. We just need to verify that
$e=d_k$. But $d_k$ divides $e$ and $\z_e$ is the eigenvalue of some
elements of $W$. So $e$ divides some $d_j$ by the remark
following~\eqref{eq:max-dim}. Therefore, $d_k$ divides $d_j$ and so
$d_k=d_j$ because $\d(d_k)=1$. This proves that $e=d_k$, as
desired.
\end{remark}

\begin{cor}\label{coro:vp-tangent}
Assume that $\d(e) =\d^*(e) \neq 0$ and let $k_0 \in
\{1,2,\dots,n\}$ be such that $d_{k_0}$ is divisible by $e$. Let $v
\in V(e)\setminus \{0\}$ and let $z=[v]$.
\begin{itemize}
\itemth{a} The family of eigenvalues of $w_e$ for its
action on the tangent space $\Trm_z(\PM(V))$ is equal to
$(\z_e^{-d_k})_{k \neq k_0}$.

\itemth{b} Let $f \in \CM[V]^W$ be homogeneous of degree $d$
and assume that $f(v)=0$. Then:
\begin{itemize}
\itemth{b1} If $d \not\equiv d_k \mod e$ for all $k \neq k_0$, then
$\ZC(f)$ is singular at $z$.

\itemth{b2} Assume that $\ZC(f)$ is smooth at $z$ and
let $k_1 \neq k_0$ be such that $d \equiv d_{k_1} \mod e$ $($the
existence of $k_1$ is guaranted by~$($\emph{b1}$))$. Then the family of
eigenvalues of $w_e$ for its action on the tangent space
$\Trm_z(\ZC(f))$ is equal to $(\z_e^{-d_k})_{k \neq k_0,k_1}$.
\end{itemize}
\end{itemize}
\end{cor}

\begin{proof}
By permuting if necessary the degrees, we may assume that $k_0=1$.
Note that $\z_e^{1-d_1}=\z_e$. Choose a basis $(v_1,\dots,v_n)$ of
$V$ such that $v=v_1$ and $w(v_k)=\z_e^{1-d_k} v_k$ for all $k \in
\{1,2,\dots,n\}$ (see Theorem~\ref{theo:springer}(e)).

\medskip

(a) Identify $\PM(V)$ with $\PM^{n-1}(\CM)$ through the choice of
this basis. Then the action of $w_e$ is transported to
$$w_e \cdot [x_1:x_2:\cdots:x_n]=
[\z_e x_1:\z_e^{1-d_2}x_2:\cdots:\z_e^{1-d_n} x_n]
=[x_1:\z_e^{-d_2}x_2:\cdots:\z_e^{-d_n} x_n].$$ Since $z=[1 : 0 :
\cdots : 0]$, this shows~(a).

\medskip

(b) Let us work in the affine chart ``$x_1=1$'', identified with
$\Ab^{n-1}(\CM)$ through the coordinates $(x_2,\dots,x_n)$. The
equation of the tangent space $\Trm_z(\ZC(f))$ is given in this
chart by
$$\sum_{k=2}^n (\partial_{v_k}f)(v) x_k =0.$$
By~(\ref{eq:derive}),
$$\lexp{w}{(\partial_{v_k} f)}(v)=\z_e^{1-d_k} (\partial_{v_k} f)(v).$$
But
$$\lexp{w}{(\partial_{v_k} f)}(v)=(\partial_{v_k} f)(w^{-1}(v))
=(\partial_{v_k} f)(\z_e^{-1} v).$$ As $\partial_{v_k} f$ is
homogeneous of degree $d-1$, this implies that
$$\z_e^{1-d_k} (\partial_{v_k} f)(v) = \z_e^{1-d}(\partial_{v_k} f)(v).$$
Therefore, if $d \not\equiv d_k \mod e$ for all $k \in
\{2,\dots,n\}$, we get $(\partial_{v_k} f)(v)=0$ for all $k \in
\{2,\dots,n\}$, and so $\ZC(f)$ is singular ar $[v]$. This shows
(b1).

Now, if $\ZC(f)$ is smooth at $z$, then there exists $k_1 \in
\{2,\dots,n\}$ such that $\partial_{v_{k_1}} f(v) \neq 0$, and in
particular $d\equiv d_{k_1} \mod e$. Then there exists $k \in
\{2,\dots,n\}$ such that $\partial_{v_k} f(v) \neq 0$. This shows
that the action of $w_e$ on the one-dimensional space
$\Trm_z(\PM(V))/\Trm_{[v]}(\ZC(f))$ is given by multiplication by
$\z_e^{-d_{k_1}}=\z_e^{-d}$. The proof of~(b2) is complete.
\end{proof}

\section{Determining singularities}\label{singu}

An important step for analyzing the properties of the K3 surfaces
constructed in the next section is to determine the singularities of
the variety $\ZC(f)/\G$ in the cases we are interested in (here, $f$
is a fundamental invariant of $W$ and $\G$ is a subgroup of $W$). We
provide in this section two different tools that will be used in the
sequel to this paper~\cite{bonnafe-sarti-2},~\cite{bonnafe-sarti-3},
where particular examples will be studied.

\subsection{Stabilizers}\label{sub:stab}
The singularity of $\ZC(f)/\G$ at the $\G$-orbit of $z \in \ZC(f)$
depends on the singularity of $\ZC(f)$ at $z$ and the action of
$\G_z$ on this (eventually trivial) singularity. We investigate here
some facts about the stabilizers $W_z$ and their action on the
tangent space $\Trm_z(\ZC(f))$.

\medskip

Let $f$ denote a homogeneous invariant of $W$, let $d$ denote its
degree and let $v \in V \setminus\{0\}$ be such that $f(v)=0$. We
set $z=[v] \in \ZC(f) \subset \PM(V)$. We denote by $\theta_z : W_z
\longto \CM^\times$ the linear character defined by
$w(v)=\theta_z(w)v$ for all $w \in W_z$. Then $W_v=\Ker(\theta_z)$
and we denote by $e_z=|{\mathrm{Im}}(\theta_z)|$. So there exists $w
\in W_z$ such that $\theta_z(w)=\z_{e_z}$. In other words, $v \in
V(w,\z_{e_z})$ and so, by Theorem~\ref{theo:springer}(a), we may,
and we will, assume that $v \in V(w_{e_z},\z_{e_z})=V(e_z)$. This
shows that \equat\label{eq:inertie} W_z=W_v \langle w_{e_z} \rangle.
\endequat
Recall from~\S\ref{sub:reflecting} that $W_v$ is a parabolic
subgroup of $W$ and so is generated by reflections. Note the
following useful facts:
\begin{itemize}
\itemth{a} Let $m=|\Zrm(W)|$
(recall from~(\ref{eq:centre}) that $m={\mathrm{gcd}}(\Deg(W))$).
Since $\mub_m=\Zrm(W) \subset W_z$, $m$ divides $e_z$.

\itemth{b} If $\d(e_z)=\d^*(e_z)$, $f(v)=0$ and $\ZC(f)$ is smooth at $z$,
then the eigenvalues of $w_{e_z}$ on the tangent space
$\Trm_z(\ZC(f))$ are given by Corollary~\ref{coro:vp-tangent}.

\itemth{c} Let $P$ be a parabolic subgroup of $W$ of rank
$n-2$ and assume that $\ZC(f)$ is {\it smooth}. Then $\dim V^P=2$
and so $L = \PM(V^P)$ is a line in $\PM(V)$. Then $L$ intersects
$\ZC(f)$ transversally by Corollary~\ref{coro:transverse}, so $|L
\cap \ZC(f)|=d$ because $f$ has degree $d$. Moreover,
\equat\label{eq:stab} \text{\it If $z \in L \cap \ZC(f)$, then
$W_v=P$.}
\endequat
Indeed, $P \subset W_v$ by construction and, if this inclusion is
strict, this means that $W_v$ has rank $n-1$ or $n$. But it cannot
have rank $n$ for otherwise $W_v=W$ and $v=0$ (which is impossible).
And it cannot have rank $n-1$ because
Corollary~\ref{coro:v-fixe-singulier} would imply that $\ZC(f)$ is
singular at $z$, contrarily to the hypothesis. This implies for
instance that two smooth points in $L \cap \ZC(f)$ are in the same
$\G$-orbit if and only if they are in the same $\Nrm_\G(P)$-orbit.

Moreover, in this case, we have a $P$-equivariant isomorphism
\equat\label{eq:tangent equivariant} \Trm_{[v]}(\ZC(f)) \simeq V/V^P
\endequat
(see the proof of Corollary~\ref{coro:transverse}).
\end{itemize}

\subsection{Singularities of double covers}
If $n=4$, $\G=W^\slv$ and $W$ is generated by reflections of order
$2$, then the surface $\ZC(f)/\G$ is the double cover of a weighted
projective plane. Most of (but not all) the singularities of
$\ZC(f)/\G$ may be then analyzed through the singularities of the
branch locus of this cover.

So we fix a double cover $\pi : Y \to X$ between two irreducible
algebraic surfaces and we assume that $Y$ is normal and $X$ is
smooth. By the purity of the branch locus, the branch locus $R$ of
$\pi$ is empty or pure of codimension $1$ (i.e. pure of dimension
$1$). The next well-known fact (see for
instance~\cite[Part~III,~\S{7}]{BPV}) will help us in our explicit
computations:

\begin{pro}\label{prop:double cover}
Let $y \in Y$ be such that $x=\pi(y)$ belongs to $R$. We assume that
$x$ is an ADE curve singularity of the branch locus $R$. Then $y$ is
an ADE surface singularity of the same type.
\end{pro}

\section{Invariant K3 surfaces}\label{invK3}

\boitegrise{{\bf Hypothesis.} {\it In this section, and only in this
section, we assume that $n=4$ and that $W$ is irreducible and
generated by reflections of order $2$.}}{0.8\textwidth}

\subsection{Classification}
We provide in Table~\ref{table:degres} the list of irreducible
complex reflection groups $W$ of rank $4$ which are generated by
reflections of order $2$ together with the following informations:
the order of $W$, the order of $W/\Zrm(W)$ (which is the group that
acts faithfully on $\PM(V)$), the order of $W'$ and the lists of
degrees and codegrees. We also recall their notation in
Shephard-Todd classification~\cite{shephardtodd} as well as their
Coxeter name whenever they are real.

\begin{table}
$$\begin{array}{!{\vline width 2pt} c !{\vline width 2pt} c|c|c|c!{\vline width 2pt}}
\hlinewd{2pt}
 W & |W| & |W/\Zrm(W)| & |W'| &
\begin{array}{c}\Deg(W) \\ \Codeg(W) \end{array} \\
\hlinewd{2pt} G(1,1,5) =\Wrm(A_4) \simeq \SG_5 & 120 & 120 & 60 &
\begin{array}{c}2,3,4,5 \\ 0,1,2,3 \end{array} \\
\hline G(e,e,4), e \ge 2 & 24e^3 & 24e^3/{\mathrm{gcd}}(e,4) & 12e^3
&
\begin{array}{c} e,2e,3e,4 \\ 0,e,2e,3e-4 \end{array} \\
\hline G(2e,e,4), e \ge 1 & 384e^3 & 192e^3/{\mathrm{gcd}}(e,4) &
96e^3 &
\begin{array}{c}2e,4e,6e,8 \\ 0,2e,4e,6e \end{array} \\
\hline
G_{28}=\Wrm(F_4) & 1\,152 & 576 & 288 &
\begin{array}{c}2, 6, 8, 12 \\ 0,4,6,10 \end{array} \\
\hline G_{29} & 7\,680 & 1\,920 & 3\,840 &
\begin{array}{c}4, 8, 12, 20 \\ 0,8,12,16\end{array} \\
\hline G_{30}=\Wrm(H_4) & 14\,400 & 7\,200 & 7\,200 &
\begin{array}{c}2, 12, 20, 30 \\ 0, 10, 18, 28\end{array} \\
\hline G_{31} & 46\,080 & 11\,520 & 23\,040 &
\begin{array}{c}8, 12, 20, 24 \\ 0, 12, 16, 28 \end{array} \\
\hlinewd{2pt}
\end{array}
$$
\refstepcounter{theorem} \caption{Irreducible complex reflection
groups of rank $4$ generated by reflections of order
$2$}\label{table:degres}
\end{table}

Recall that $G(2,1,4)=\Wrm(B_4)$ and $G(2,2,4)=\Wrm(D_4)$. Note that
the hypothesis on the order of the reflections implies that
\equat\label{eq:wsl-2} |W^\slv|=\frac{|W|}{2}.
\endequat
In particular, $W' \neq W^\slv$ if and only if $W=G_{28}$ or
$W=G(2e,e,4)$ for some $e \ge 1$. Also, note the following diagram
of non-trivial inclusions between those of the complex reflection
groups which are contained in a primitive one (here, $H
\stackrel{\vartriangleleft}{\longinjto} G$ means that $H$ is a
normal subgroup of $G$). \equat\label{eq:inclusions}
\begin{array}{c}\xymatrix{
& G(2,1,4) \xyinj[r] \xyinj[dr] & G_{29} \xyinj[dr]\\
G(2,2,4) \xyinj[dr]\xyinj[ur]^{\vartriangleleft}
\xyinj[rr]^{\vartriangleleft}
& & G_{28} \xyinj[r] & G_{31} \\
& G(4,4,4) \xyinj[r]^{\vartriangleleft} & G(4,2,4) \xyinj[ur] }
\end{array}
\endequat

\subsection{K3 surfaces}\label{sub:k3}
Equations of surfaces of the form $\ZC(f)/W'$ or $\ZC(f)/W^\slv$
(where $f$ is a fundamental invariant of degree $d$) in a weighted
projective space are provided by Propositions~\ref{prop:zf-dw}
and~\ref{prop:zf-wsl}. Whenever some arithmetic conditions on $d$
and the degrees of $W$ are satisfied, it can then be proven thanks
to results of Appendix~\ref{app:2-forme} (and particularly
Corollary~\ref{coro:canonique}) that the canonical sheaf of
$\ZC(f)/W'$ or $\ZC(f)/W^\slv$ is trivial (provided that $\ZC(f)$ is
normal, so that the quotient is also normal and the canonical sheaf
is well-defined): it turns out that, in most cases, the quotient
$\ZC(f)/W'$ or $\ZC(f)/W^\slv$ has only ADE singularities and
positive Euler characteristic so that their minimal resolution are
K3 surfaces. A particular feature of these examples is that their
minimal resolution have always a big Picard number, as big as
possible compare to the number of moduli of the family. Note that
some of these examples were already studied by Barth and the second
author~\cite{barth-sarti}: we revisit these cases and simplify some
arguments using more theory about complex reflection groups.

We denote by $\KC_3$ the set of pairs $(W,d)$ where $W$ is an
irreducible complex reflection group of rank $4$ and $d$ is a
positive integer satisfying one of the following conditions:
\begin{itemize}
\item $W=G(1,1,5)$, $G(4,2,4)$ or $G_{29}$,
and $d=4$.

\item $W=G(2e,2e,4)$ with $e$ odd, and $d \in \{4e,6e\}$.

\item $W=G(4e,4e,4)$, and $d =4e$.

\item $W=G(2,1,4)$, and $d \in \{4,6\}$.

\item $W=G_{28}$, and $d \in \{6,8\}$.

\item $W=G_{30}$, and $d=12$.

\item $W=G_{31}$, and $d=20$.
\end{itemize}

\begin{theorem}\label{theo:main}
Assume that $(W,d) \in \KC_3$. Let $\G$ be the subgroup $W'$ or
$W^\slv$ of $W$ and let $f$ be a fundamental invariant of $W$ of
degree $d$ such that $\ZC(f)$ has only ADE singularities.

Then $\ZC(f)/\G$ is a K3 surface with ADE singularities.
\end{theorem}

The proof of Theorem~\ref{theo:main} will be given
in~\S\ref{sub:proof} and Section~\ref{sec:singularites}. As an
immediate consequence, we get:

\begin{cor}\label{coro:main}
Under the hypotheses of Theorem~\ref{theo:main}, the minimal
resolution of $\ZC(f)/\G$ is a smooth projective K3 surface.
\end{cor}

\subsection{Numerical informations}
Before proving this Theorem~\ref{theo:main}, let us make some
remarks. By Propositions~\ref{prop:zf-dw} and~\ref{prop:zf-wsl}, the
variety $\ZC(f)/\G$ is a weighted complete intersection
(see~\cite[\S{3.2}]{dolga_weighted} for the definition) in a
weighted projective space (it is defined by one or two equations).
If $\G=W^\slv$, then $\ZC(f)/\G$ is a weighted hypersurface in a
weighted projective space of dimension $3$ (see
Proposition~\ref{prop:zf-wsl}). If $\G=W'$, then $\ZC(f)/\G$ is a
codimension $2$ weighted complete intersection in a weighted
projective space of dimension $4$ (see
Proposition~\ref{prop:zf-dw}). We give in Table~\ref{table:eq-k3}
the list of the weights of the ambient projective space as well as
the list of the degrees of the equations in all the different cases
(we also give the description of $\ZC(f)/W$ as a weighted projective
space).

By looking at this Table~\ref{table:eq-k3}, the reader might think
that we have build infinitely many families of K3 surfaces, by
letting the integer $e$ vary in the fourth and fifth group
considered. However, as it will be explained
in~\S\ref{sub:imprimitif} (see the
isomorphisms~(\ref{eq:6e}),~(\ref{eq:4e-1}) and~(\ref{eq:4e-2})),
the general group with parameter $e$ and the particular group for
$e=1$ give exactly the same families of surfaces.

Also, it turns out that $G(2,2,4)'=G(2,1,4)'$, and since invariants
of $G(2,1,4)=\Wrm(B_4)$ are invariants for $G(2,2,4)=\Wrm(D_4)$,
this shows that two of the families of K3 surfaces constructed with
$G(2,1,4)$ are contained in families built from $G(2,2,4)$: note
however that, for these particular examples, having the two points
of view give different embeddings in weighted projective spaces.

As a consequence, we have build 15 families of K3 surfaces (note
that the families corresponding to the groups $G(4,2,4)$ and
$G_{29}$ are $0$-dimensional, as there is, up to scalar, a single
quartic polynomial invariant by each of these groups). If we exclude
the ``easy'' case of the quotient of a quartic by a finite subgroup
of $\Sb\Lb_4(\CM)$ (see~\S\ref{sub:d=4}), it remains $8$ non-zero
dimensional families of K3 surfaces whose construction is
non-trivial.

\begin{center}{\begin{table}
$$\begin{array}{!{\vline width 2pt} c !{\vline width 2pt} c|c|c|c|c!{\vline width 2pt}}
\hlinewd{2pt}
 W & d & \G &  \text{Ambient space}
 &  \begin{array}{c} \text{degree(s) of} \\ \text{equation(s)} \end{array} & \ZC(f)/W \\
\hlinewd{2pt} \petitespace \Wrm(A_4) \simeq \SG_5 & 4 & W'=W^\slv &
\PM(2,3,5,10) &
20 & \PM(2,3,5) \\
\hlinewd{0.8pt} \multirow{4}{*}{$\begin{array}{c} G(2,1,4) \\
=\Wrm(B_4)\end{array}$} &
\multirow{2}{*}{4} & W' & \petitespace \PM(1,3,4,6,2) & 12,4 & \multirow{2}{*}{$\PM(1,3,4)$}\\
\clinewd{0.1pt}{3-5} && W^\slv & \petitespace \PM(1,3,4,8) & 16 &\\
\clinewd{0.3pt}{2-6} & \multirow{2}{*}6 & W' & \petitespace
\PM(1,1,2,3,1) & 6,2
& \multirow{2}{*}{$\PM(1,1,2)$}\\
\clinewd{0.1pt}{3-5} && W^\slv & \petitespace \PM(1,1,2,4) & 8 & \\
\hlinewd{0.8pt} \petitespace \multirow{2}{*}{$G(4,2,4)$} &
\multirow{2}{*}{$4$}
& W' & \PM(2,3,2,1,6) & 2,12 & \multirow{2}{*}{$\PM(1,3,1)$} \\
\clinewd{0.1pt}{3-5} && W^\slv & \petitespace \PM(2,3,2,7) & 14 &\\
\hlinewd{0.8pt} \multirow{2}{*}{$\begin{array}{c} G(2e,2e,4) \\
\text{$e$ odd}\end{array}$} &
4e & W'=W^\slv & \petitespace \PM(1,3,2,6) & 12 & \PM(1,3,2) \\
\clinewd{0.3pt}{2-6}
& 6e & W'=W^\slv & \petitespace \PM(1,1,1,3) & 6 & \PM(1,1,1) \\
\hlinewd{0.8pt} G(4e,4e,4) &
4e & W'=W^\slv & \petitespace \PM(2,3,1,6) & 12 & \PM(2,3,1) \\
\hlinewd{0.8pt} \multirow{4}{*}{$\begin{array}{c} G_{28} \\
=\Wrm(F_4)\end{array}$} &
\multirow{2}{*}{6} & W' & \petitespace \PM(1,2,3,3,3) & 6,6 & \multirow{2}{*}{$\PM(1,2,3)$}\\
\clinewd{0.1pt}{3-5} && W^\slv & \petitespace \PM(1,2,3,6) & 12 & \\
\clinewd{0.3pt}{2-6} & \multirow{2}{*}8 & W' & \petitespace
\PM(1,1,2,2,2) & 4,4
& \multirow{2}{*}{$\PM(1,1,2)$}\\
\clinewd{0.1pt}{3-5} && W^\slv & \petitespace \PM(1,1,2,4) & 8 & \\
\hlinewd{0.8pt}
G_{29} & 4 & W'=W^\slv & \petitespace \PM(2,3,5,10) & 20 & \PM(2,3,5)\\
\hlinewd{0.8pt}
G_{30}=\Wrm(H_4) & 12 & W'=W^\slv & \petitespace \PM(1,2,3,6) & 12 & \PM(1,2,3)\\
\hlinewd{0.8pt}
G_{31} & 20 & W'=W^\slv & \petitespace \PM(2,1,2,5) & 10 & \PM(1,1,1)\\
\hlinewd{2pt}
\end{array}
$$
\refstepcounter{theorem} \caption{Weights of ambient projective
spaces and degree(s) of equation(s) of
$\ZC(f)/\G$}\label{table:eq-k3}
\end{table}}\end{center}

\boitegrise{{\bf Hypothesis and notation.} {\it From now on, and
until the end of this paper, we assume that $(W,d) \in \KC_3$, that
$\G$ is the subgroup $W'$ or $W^\slv$ of $W$, and that $f$ is a
fundamental invariant of $W$ of degree $d$ such that $\ZC(f)$ has
only ADE singularities. We also fix a family $\fb=(f,f_1,f_2,f_3)$
of fundamental invariants containing $f$ and we set
$d_i=\deg(f_i)$.}}{0.8\textwidth}

\begin{proof}[Proof of the results given in Table~\ref{table:eq-k3}]
The proof follows from Table~\ref{table:degres} and a case-by-case
analysis. We will not give details for all cases, we will only treat
two cases (the reader can easily check that all other cases can be
treated similarly).

\medskip

$\bullet$ Assume that $(W,d)=(G_{31},20)$ and that $\G=W'$
($=W^\slv$). Then $(d_1,d_2,d_3)=(8,12,24)$ by
Table~\ref{table:degres} and $|\AC|=60$ by~(\ref{eq:degres})
and~(\ref{eq:ref-hyp}). It then follows from
Proposition~\ref{prop:zf-wsl} that
$$\ZC(f)/\G=
\{[x_1:x_2:x_3:j] \in
\PM(8,12,24,60)~|~j^2=P_\fb(0,x_1,x_2,x_3)\}.$$ But
$\PM(8,12,24,60)=\PM(2,3,6,15)=\PM(2,1,2,5)$, so that
$$\ZC(f)/\G = \{[y_1:y_2:y_3:j] \in \PM(2,1,2,5)~|~j^2 =
Q(y_1,y_2,y_3)\}$$ for some polynomial $Q \in \CM[Y_1,Y_2,Y_3]$.
Hence $\ZC(f)/\G$ is defined by an equation of degree $10$ in
$\PM(2,1,2,5)$, as expected.

Finally, by Proposition~\ref{prop:zf-wsl}, $\ZC(f)/W \simeq
\PM(8,12,24) =\PM(2,3,6) \simeq \PM(2,1,2) \simeq \PM(1,1,1)$, as
expected.

\medskip

$\bullet$ Assume that $(W,d)=(G_{28},6)$ and $\G=W'$. Then
$(d_1,d_2,d_3)=(2,8,12)$ by Table~\ref{table:degres}. Note that $W'
\neq W^\slv$ and that there are two $W$-orbits $\O_1$ and $\O_2$ of
reflecting hyperplanes, which are both of cardinality $12$. It then
follows from Proposition~\ref{prop:zf-dw}, that
$$\ZC(f)/\G=\{[x_1:x_2:x_3:j_1:j_2] \in \PM(2,8,12,12,12)~|~
j_1^2=P_{\fb,\O_1}(0,x_1,x_2,x_3)$$
$$\hphantom{AAAAAAAAAAAAAAAA}~\text{and}~j_2^2=P_{\fb,\O_2}(0,x_1,x_2,x_3)\}.$$
But $\PM(2,8,12,12,12)=\PM(1,4,6,6,6)=\PM(1,2,3,3,3)$, so that
$$\ZC(f)/\G=\{[y_1:y_2:y_3:j_1:j_2] \in \PM(1,2,3,3,3)~|~
j_1^2=Q_{1}(y_1,y_2,y_3)~\text{and}~j_2^2=Q_{2}(y_1,y_3,y_4)\}$$ for
some polynomials $Q_1$ and $Q_2$ in $\CM[Y_1,Y_2,Y_3]$. So
$\ZC(f)/\G$ is defined by two equations of degree $6$ in
$\PM(1,2,3,3,3)$, as expected.

Finally, by Proposition~\ref{prop:zf-dw}, $\ZC(f)/W \simeq
\PM(2,8,12) =\PM(1,4,6) \simeq \PM(1,2,3)$, as expected.
%
%
%
\end{proof}

\begin{remark}\label{rem:2}
The arithmetic of degrees and the classification of reflection
groups imply that it does not seem possible to find a complex
reflection $W$ and a degree $d$ of $W$ such that $W$ is not
generated by reflections of order $2$ and $\ZC(f)/\G$ has a trivial
canonical sheaf, except whenever $d=4$. But this is in some sense
the less exciting case, as it is shown by the argument given
in~\S\ref{sub:d=4} below.

Also, note that if $e \not\in \{1,2,4\}$, then $G(e,e,4)$ has a
unique invariant of degree $4$ that defines a quartic in
$\PM^3(\CM)$, but this invariant is equal to $xyzt$, and so $\ZC(f)$
is not irreducible and does not fulfill the hypothesis of
Theorem~\ref{theo:main}. That is why this case does not appear in
the list $\KC_3$.
\end{remark}

\begin{remark}\label{rem:barth-sarti}
If $(W,d)=(G_{28},6)$ or $(G_{30},12)$, and $\G=W'$, then the above
result was obtained by Barth and the second
author~\cite{barth-sarti}: the group $\G$ was denoted by $G_d$ in
their paper (this must not be confused with Shephard-Todd
notation).
\end{remark}

\subsection{About the families attached to
${\boldsymbol{G(2e,2e,4)}}$}\label{sub:imprimitif} Assume in this
subsection, and only in this subsection, that $W=G(2e,2e,4)$ for
some $e$. Recall that $G(2e,2e,4)$ is the group of monomial matrices
in $\Gb\Lb_4(\CM)$ with coefficients in $\mub_{2e}$ and such that
the product of the non-zero coefficients is equal to $1$. Note that
this implies that $W'=W^\slv$.

If $1 \le k \le 4$, we denote by $\s_k$ the $j$-th elementary
symmetric function in the variables $x$, $y$, $z$, $t$, and let
$$J_1=(x-y)(x-z)(x-t)(y-z)(y-t)(z-t).$$
If $p \in \CM[x,y,z,t]$ and $l \ge 1$ is an integer, we set
$p[l]=p(x^l,y^l,z^l,t^l) \in \CM[x,y,z,t]$. For instance,
$\s_1[l]=\Sigma(x^l)$. Then $(\s_1,\s_1[2],\s_1[3],\s_4)$ is a
family of fundamental invariants of $G(1,1,4) \simeq \SG_4$. So
there exists a unique polynomial $P \in \CM[x_1,x_2,x_3,x_4]$ such
that \equat\label{eq:j-sn} J_1^2=P(\s_1,\s_1[2],\s_1[3],\s_4).
\endequat
We do not need here the explicit form of $P$. If $a$, $b$, $c \in
\CM$, we set
$$F_{a,b,c}=a \s_1[4] + b \s_4 + c \s_1[2]^2
= a\Sigma(x^4) + b xyzt + c (\Sigma(x^2))^2,$$
$$G_{a,b,c}=\s_1[6]+a\s_1[4]\s_1[2]+b\s_1[2]^3+c\s_1[2]\s_4.$$
With this notation, we may, and we will, choose as a family of
fundamental invariants of $W$ the family
$(\s_1[2e],\s_1[4e],\s_1[6e],\s_4)$. Note that the element $J \in
\CM[x,y,z,t]^{W'}$ defined in~\S\ref{sub:reflecting} is equal to
$J_1[2e]$ (up to a scalar). Applying the endomorphism $p \mapsto
p[2e]$ of $\CM[x,y,z,t]$ to the formula~(\ref{eq:j-sn}) and using
Proposition~\ref{prop:zf-wsl} gives \equat\label{eq:pv-g2e}
\PM(V)/W' \!\simeq\! \{[x_1\!:\!x_2\!:\!x_3\!:\!x_4\!:\!j] \in
\PM(2e,4e,6e,4,12e)| j^2\!=\!P(x_1,x_2,x_3,x_4^{2e})\},
\endequat
because $\s_4[2e]=\s_4^{2e}$. Let us examine the situation according
to the parity of $e$.

\subsubsection{The case where $e$ is odd}
Assume here that $e$ is odd. Then
$$\PM(2e,4e,6e,4,12e)=\PM(e,2e,3e,2,6e)\simeq \PM(1,2,3,2,6).$$
So it follows from~(\ref{eq:pv-g2e}) that \equat\label{eq:pv-g2e-impair}
\PM(V)/W' \simeq \{[x_1:x_2:x_3:x_4:j] \in \PM(1,2,3,2,6)~|~
j^2=P(x_1,x_2,x_3,x_4^2)\}.
\endequat
Now, a fundamental invariant of degree $6e$ of $W$ if of the form
$G_{a,b,c}[e]$. We deduce from~(\ref{eq:pv-g2e-impair}) that
$$\ZC(G_{a,b,c}[e])/W' \simeq \{[x_1:x_2:x_4:j] \in \PM(1,2,2,6)~|~
j^2=P(x_1,x_2,-ax_1x_2-bx_1^3-cx_1x_4,x_4^2)\}$$ and in particular
\equat\label{eq:6e}
\ZC(G_{a,b,c}[e])/G(2e,2e,4)'=\ZC(G_{a,b,c})/G(2,2,4)'.
\endequat
Similarly, a fundamental invariant of degree $4e$ of $W$ is of the
form $F_{a,b,c}[e]$ with $(a,b) \neq 0$ and \equat\label{eq:4e-1}
\ZC(F_{a,b,c}[e])/G(2e,2e,4)'=\ZC(F_{a,b,c})/G(2,2,4)'.
\endequat
This shows that the varieties $\ZC(G_{a,b,c}[e])/G(2e,2e,4)'$ and
$\ZC(F_{a,b,c}[e])/G(2e,2e,4)'$ do not depend on $e$.

\subsubsection{The case where $e$ is even}
Assume here that $e=2e'$ for some $e' \ge 1$. Then
$$\PM(2e,4e,6e,4,12e)=\PM(e',2e',3e',1,6e') = \PM(1,2,3,1,6).$$
So it follows from~(\ref{eq:pv-g2e}) that \equat\label{eq:pv-g2e-pair}
\PM(V)/W' \simeq \{[x_1:x_2:x_3:x_4:j] \in \PM(1,2,3,1,6)~|~
j^2=P(x_1,x_2,x_3,x_4^4)\}.
\endequat
Now, a fundamental invariant of $G(4e',4e',4)$ of degree $4e'$ is of
the form $F_{a,b,c}[e']$ and a similar argument as before shows that
\equat\label{eq:4e-2}
\ZC(F_{a,b,c}[e'])/G(4e',4e',4)'=\ZC(F_{a,b,c})/G(4,4,4)'.
\endequat
Again, the variety $\ZC(F_{a,b,c}[e'])/G(4e',4e',4)'$ does not
depend on $e'$.

\subsubsection{Complements}
Note for future reference (see~\S\ref{sub:g2e}) the following fact:

\begin{lemma}\label{lem:F-a1}
If $\ZC(F_{a,b,c})$ is irreducible, then it is smooth or has only
$A_1$ singularities.
\end{lemma}

\begin{proof}
Assume that $\ZC(F_{a,b,c})$ is irreducible and singular. Let us
first assume that $a = 0$. Then we may assume that $b=1$ and the
irreducibility of $\ZC(F_{0,1,c})$ forces $c \neq 0$. An easy
computation then shows that the only singular points of
$\ZC(F_{0,1,c})$ are the ones belonging to the $G(2,2,4)$-orbit of
$p=[0 : 0 : i : 1]$. But the homogeneous component of degree $2$ of
$F_{0,1,c}(x,y,i+z,1)$ is $ixy-4z^2$, which is a non-degenerate
quadratic form in $x$, $y$, $z$. So $p$ is an $A_1$ singularity of
$\ZC(F)$, as expected.

\smallskip

Le us now assume that $a \neq 0$, and even that $a=1$. Assume that
we have found $(b_0,c_0) \in \CM^2$ such that $\ZC(F_{1,b_0,c_0})$
admits a singular point $q=[x_0 : y_0 : z_0 : t_0]$ which is not an
$A_1$ singularity. Then, by permuting the coordinates if necessary,
we may assume that $t_0 \neq 0$. So let us work in the affine chart
$t \neq 0$ and set $F_{b,c}^\circ=F_{1,b,c}(x,y,z,1)$. Let $H_{b,c}$
denote the Hessian matrix of $F_{b,c}^\circ$. Then
$(x_0,y_0,z_0,b_0,c_0)$ belongs to the variety
$$\XC=\{(x,y,z,b,c) \in \AM^5(\CM)~|~
F^\circ_{b,c}(x,y,z)= \frac{\partial F^\circ_{b,c}}{\partial
x}(x,y,z)=$$
$$\frac{\partial F^\circ_{b,c}}{\partial y}(x,y,z)=
\frac{\partial F^\circ_{b,c}}{\partial z}(x,y,z)=
\det(H_{b,c}(x,y,z))=0\}.
$$
Now let $\pi : \AM^5(\CM) \to \AM^2(\CM)$, $(x,y,z,\beta,\g) \mapsto
(\beta,\g)$.
A {\sc Magma}~\cite{magma} 
computation:
\begin{quotation}{\small\begin{verbatim}
> A5<x,y,z,b,c>:=AffineSpace(Rationals(),5);
> A2<b0,c0>:=AffineSpace(Rationals(),2);
> pi:=map<A5 -> A2 | [b,c]>;
> Fbco:=(x^4+y^4+z^4+1)+b*x*y*z+c*(x^2+y^2+z^2+1)^2;
> Hbc:=Matrix(CoordinateRing(A5),3,3,
>   [[Derivative(Derivative(Fbco,k),l) :
>     k in [1..3]] : l in [1..3]]);
> X:=Scheme(A5,[Fbco] cat [Derivative(Fbco,k) : k in [1,2,3]]
>   cat [Determinant(Hbc)]);
> MinimalBasis(ReducedSubscheme(pi(X)));
[
    c0 + 1/2,
    b0^2 - 16
]
\end{verbatim}
}\end{quotation}
shows that $\pi(\XC)=\{(4,-1/2),(-4,-1/2)\}$. This implies that $(b_0,c_0) \in
\{(4,-1/2),(-4,-1/2)\}$. But $F_{4,-1/2}$ and $F_{-4,-1/2}$ are not
irreducible (they are divisible by $x-y-z+t$ and $x+y+z-t$
respectively): this contradicts the hypothesis.
\end{proof}

\begin{remark}
Observe that the previous family $\ZC(F_{a,b,c}$) with $(a,b)=(1,0)$
was studied in \cite{sarti_comm}, where it is shown that the family
contains exactly four singular surfaces with $4, 8, 12, 16$
$A_1$--singularities. In particular the surface with $16$
$A_1$--singularities is a Kummer K3 surface.
\end{remark}

\subsection{Proof of Theorem~\ref{theo:main}}\label{sub:proof}
The rest of this paper is devoted to the proof of this
Theorem~\ref{theo:main}. Note that it proceeds by a case-by-case
analysis, but this case-by-case analysis is widely simplified by the
general facts about complex reflection groups recalled in the
previous sections.

\begin{proof}[Proof of Theorem~\ref{theo:main}]
Assume that the hypotheses of Theorem~\ref{theo:main} are satisfied.
For proving that $\ZC(f)/\G$ is a K3 surface with ADE singularities,
we need to show the following facts:
\begin{itemize}
\itemth{S} The surface $\ZC(f)/\G$ has only ADE singularities.

\itemth{E} The Euler characteristic of $\ZC(f)/\G$ is positive.

\itemth{C} The canonical divisor of $\ZC(f)/\G$ is trivial.
\end{itemize}
Indeed, if $\Xti$ denotes the minimal resolution of $\ZC(f)/\G$ and
if~(S),~(E) and~(C) are proved, then $\Xti$ has a trivial canonical
sheaf by~(S) and~(C), so by the classification of smooth algebraic
surfaces $\Xti$ is a K3 surface or an abelian surface. But, by~(S),
the Euler characteristic of $\Xti$ is greater than or equal to the
one of $\ZC(f)/\G$, so is also positive by~(E). Since the Euler
characteristic of an abelian surface is $0$, we deduce that $\Xti$
is a smooth K3 surface.

\medskip

The technical step is to prove~(S), namely that $\ZC(f)/\G$ has only
ADE singularities. This will be postponed to the next
Section~\ref{sec:singularites}. So assume here that~(S) is proved.

Let us now prove the statement~(E), namely that the Euler
characteristic of $\ZC(f)/\G$ is positive. Since $\ZC(f)$ has only
isolated singularities by~(S), it follows
from~\cite[Theorem~4.3]{dimca} that $\Hrm^1(\ZC(f),\CM)=0$. Since
$\ZC(f)$ has only ADE singularities, it is rationally
smooth~\cite[Definition~A1]{KL}. As it is also projective, one can
apply Poincar\'e duality and so $\Hrm^3(\ZC(f),\CM)$ is the dual of
$\Hrm^1(\ZC(f),\CM)$, hence is equal to $0$. So $\ZC(f)$ has no odd
cohomology and since $\Hrm^j(\ZC(f)/\G,\CM)=\Hrm^j(\ZC(f),\CM)^\G$,
this shows that $\ZC(f)/\G$ has no odd cohomology. So its Euler
characteristic is positive.

Now it remains to prove~(C), namely that the canonical sheaf of
$X=\ZC(f)/\G$ is trivial. For this, we use
Corollary~\ref{coro:canonique}, so we need to prove that $X$
satisfies the hypotheses~(H1),~(H2),~(H3),~(H4) and~(H5) of
Appendix~\ref{app:2-forme}. Statements~(H1),~(H2) and~(H4) are
easily checked thanks to Table~\ref{table:eq-k3} while~(H5) follows
from~(S). So it remains to prove~(H3), namely that $X$ is a
well-formed weighted complete intersection. There are two cases:
\begin{itemize}
\item[$\bullet$] If $\G=W^\slv$, then $X$ is a weighted
hypersurface of degree $e$ in some $\PM(l_0,l_1,l_2,l_3)$, and,
according to~\cite[\S{6.10}]{fletcher}, $X$ is well-formed if, for
all $0 \le a < b \le 3$, $\gcd(l_a,l_b)$ divides $e$. This is easily
checked with Table~\ref{table:eq-k3}.

\item[$\bullet$] If $\G=W'\neq W^\slv$, then $X$ is a weighted
complete intersection defined by two equations of degree $e_1$ and
$e_2$ in some $\PM(l_0,l_1,l_2,l_3,l_4)$, and, according
to~\cite[\S{6.11}]{fletcher}, $X$ is well-formed if the following
two properties are satisfied:
\begin{itemize}
\item For all $0 \le a < b \le 4$, $\gcd(l_a,l_b)$ divides $e_1$ or $e_2$.

\item For all $0 \le a < b < c \le 4$, $\gcd(l_a,l_b,l_c)$ divides $e_1$ and $e_2$.
\end{itemize}
Again, this is easily checked with Table~\ref{table:eq-k3}.
\end{itemize}
The proof of Theorem~\ref{theo:main} is complete, up to the proof
of~(S).
\end{proof}

\section{Singularities of $\ZC(f)/\G$}\label{sec:singularites}

The aim of this section is to complete the proof of
Theorem~\ref{theo:main}, by proving that, under its hypotheses, the
surface $\ZC(f)/\G$ has only ADE singularities. The first subsection
is devoted to the (easy) case where $d=4$ while the third subsection
deals with the case where $(W,d)=(G(2e,2e,4),4e)$ for some odd $e$
by reduction to the $d=4$ case. The other cases are treated in the
second and fourth section.

\subsection{The case ${\boldsymbol{d=4}}$}\label{sub:d=4}
Assume in this subsection, and only in this subsection, that $d=4$.
This case is somewhat particular and requires its own treatment. It
is also well known in the literature, but we recall the discussion
for convenience of the reader. First, note that the hypothesis
implies that $\ZC(f)$ is already a K3 surface (with eventually ADE
singularities) and we denote by $\omega$ a non-degenerate global
holomorphic $2$-form on the smooth locus (it is well-defined up to a
scalar). By hypothesis, $\G \subset \Sb\Lb_4(\CM)$, so $\G$
preserves $\omega$. So $\ZC(f)/\G$ inherits a non-degenerate global
holomorphic $2$-form $\omega_\G$ on its smooth locus.

Now, let $p : X \to \ZC(f)$ denote a minimal resolution of $\ZC(f)$,
and let $\omega_X$ denote the unique non-degenerate $2$-form on $X$
extending $\omega$. Then $X$ is a K3 surface which inherits an
action of $\G$ which stabilizes $\omega$: so this is a symplectic
action and so $X/\G$ is a K3 surface with ADE singularities
\cite[\S{5}]{nikulin1}. Let $q : Y \to X/\G$ denote a minimal
resolution. We have a commutative diagram
$$\xymatrix{
& X \ar[r]^{\DS{p}} \ar[d] & \ZC(f) \ar[d] \\
Y \ar[r]^{\DS{q}} & X/\G \ar[r]^{\DS{p_\G}} & \ZC(f)/\G, }$$ where
$p_\G$ is induced by $p$. This shows that $p_\G \circ q : Y \to
\ZC(f)/\G$ is a symplectic resolution, so $\ZC(f)/\G$ is a K3
surface with ADE singularities (and $Y$ is its minimal resolution).
This completes the proof of Theorem~\ref{theo:main} whenever $d=4$.

\subsection{The case where ${\boldsymbol{\ZC(f)/W}}$ is smooth}\label{sub:smooth}
Assume in this subsection, and only in this subsection, that
$\ZC(f)/W$ is smooth. By examining Table~\ref{table:eq-k3}, this
occurs only if $(W,d)=(G(2e,2e,4),6e)$ or $(G_{31},20)$. In both
cases, $\G=W'=W^\slv$ is of index $2$ in $W$, and so the result
follows from Corollary~\ref{coro:ade}.

\subsection{The case where ${\boldsymbol{W=G(2e,2e,4)}}$}\label{sub:g2e}
Assume in this subsection, and only in this subsection, that
$W=G(2e,2e,4)$. The case where $e$ is odd and $d=6e$ is treated in
the previous subsection~\ref{sub:smooth}. If $e$ is odd and $d=4e$,
then it follows from the isomorphism~(\ref{eq:4e-1}) and
Lemma~\ref{lem:F-a1} that we may assume that $e=1$. Then $d=4$ and
this case is treated in~\S\ref{sub:d=4}. If $e=2e'$ is even and
$d=4e'$, then it follows from the isomorphism~(\ref{eq:4e-1}) and
Lemma~\ref{lem:F-a1} that we may assume that $e'=1$. Then $d=4$ and
this case is treated in~\S\ref{sub:d=4}.

\subsection{Remaining cases}
According to the cases treated
in~\S\ref{sub:d=4},~\S\ref{sub:smooth} and~\S\ref{sub:g2e}, we may
now work under the following hypothesis:

\medskip

\boitegrise{{\bf Hypothesis.} {\it From now on, and until the end of
this section, we assume that $d \neq 4$ and $W \in
\{G(2,1,4),G_{28},G_{30}\}$.}}{0.8\textwidth}

%

\subsubsection{The case where $\G=W^\slv$}\label{sub:slv}
Assume in this subsection, and only in this subsection, that
$\G=W^\slv$. In this case, $\G$ is a subgroup of index $2$ of $W$.
Recall from Propositions~\ref{prop:projectif-a-poids}
and~\ref{prop:zf-wsl} that \equat\label{eq:zf-w} \ZC(f)/W \simeq
\PM(d_1,d_2,d_3)
\endequat
and \equat\label{eq:zf-wsl} \ZC(f)/W^\slv \simeq \{[x_1:x_2:x_3:j]
\in \PM(d_1,d_2,d_3,|\AC|)~|~ j^2 = P_\fb(0,x_1,x_2,x_3)\}.
\endequat
We denote by $\r : \ZC(f)/W^\slv \longto \ZC(f)/W$ the canonical
map, let $\UC$ denote the smooth locus of $\PM(d_1,d_2,d_3)$ and let
$\SC$ denote the set of singular points of $\PM(d_1,d_2,d_3)$.

Now, $\UC=\pi_\fb^{-1}(\UC)/W$ is smooth so
$\r^{-1}(\UC)=\pi_\fb^{-1}(\UC)/\G$ contains only ADE singularities
by Corollary~\ref{coro:ade} (because $\G=W^\slv$ has index $2$ in
$W$). Hence, it remains to show that the points in $\r^{-1}(\SC)$
are smooth or ADE singularities. Let $p_1=[1:0:0]$, $p_2 = [0:1:0]$
and $p_3=[0:0:1]$ in $\PM(d_1,d_2,d_3)$. Then $\SC \subset
\{p_1,p_2,p_3\}$. The following fact is checked by a case-by-case
analysis:

\begin{lemma}\label{lem:etale}
Assume that $d \neq 4$ and $W \in \{G(2,1,4),G_{28},G_{30}\}$. If
$p_k \in \PM(d_1,d_2,d_3)$ is singular, then:
\begin{itemize}
\itemth{a} $\d(d_k)=\d^*(d_k)=1$.

\itemth{b} $\det(w_{d_k})=1$.

\itemth{c} $p_k$ is an $A_j$ singularity of $\PM(d_1,d_2,d_3)$
for some $j \ge 1$.
\end{itemize}
\end{lemma}

The proof will be given below. Let us first explain why this lemma
might help to check that the points in $\r^{-1}(\SC)$ are smooth or
ADE singularities. So let $p_k \in \SC$ and let
$\O_k=\pi_\fb^{-1}(p_k)$. Then
$$\O_k=\{p \in \PM(V)~|~\forall~1 \le j \neq k \le 3,~f_j(p)=f_k(p)=0\}.$$
By Lemma~\ref{lem:etale}, $\dim V(d_k)=1$, so we might view $V(d_k)
\in \PM(V)$ as a point $z_k \in \ZC(f)$. We denote by $z_k^\slv$ the
image of $z_k$ in $\ZC(f)/\G$. By Theorem~\ref{theo:springer}(d), we
have that $\O_k$ is the $W$-orbit of $z_k$. But the stabilizer of
$z_k$ in $W$ is $\langle w_{d_k} \rangle$ by
Remark~\ref{rem:stab-cyclique}, so it is contained in $\G$ by
Lemma~\ref{lem:etale}. So the map $\r$ is \'etale at $z_k^\slv$, and
so the singularity of $\ZC(f)/\G$ at $z_k^\slv$ is equivalent to the
singularity of $\PM(d_1,d_2,d_3)$ at $p_k$, hence is an $A_j$
singularity by Lemma~\ref{lem:etale}. This completes the proof of
Theorem~\ref{theo:main} whenever $\G=W^\slv$ and $(W,d) \neq
(G(2e,2e,4),4e)$, provided that Lemma~\ref{lem:etale} is proved.
This is done just below:

\begin{proof}[Proof of Lemma~\ref{lem:etale}]
Let us examine the different cases:
\medskip

\begin{itemize}
%
%
%
%
\item[$\bullet$] {\it Type $G(2,1,4)$.}
Assume here that $W=G(2,1,4)$. Then $d=6$ and
$(d_1,d_2,d_3)=(2,4,8)$. But
$\PM(2,4,8)\simeq\PM(1,2,4)\simeq\PM(1,1,2)$, so $\SC = \{p_3\}$ and
$p_3$ is an $A_1$ singularity. Moreover, $d_3=8$ and it follows from
Table~\ref{table:degres} that $\d(8)=\d^*(8)=1$. Also,
Theorem~\ref{theo:springer}(f) implies that the eigenvalues of
$w_{8}$ are $(\z_8^{-5},\z_8^{-1},\z_8^{-3},\z_8)$, so
$\det(w_{8})=1$.

\medskip

\item[$\bullet$] {\it Type $G_{28}$.}
Assume here that $W=G_{28}$. If $d=6$, then
$(d_1,d_2,d_3)=(2,8,12)$: but in that case $\PM(2,8,12)=\PM(1,4,6)=\PM(1,2,3)$ so
$\SC=\{p_2,p_3\}$ and so $d_k \in \{8,12\}$ (and note that $p_2$ is
an $A_1$ singularity, while $p_3$ is an $A_2$ singularity). If
$d=8$, then $(d_1,d_2,d_3)=(2,6,12)$ and we have
$\PM(2,6,12)=\PM(1,3,6)=\PM(1,1,2)$ so $\SC=\{p_3\}$ and so $d_k =
12$ (and note that $p_3$ is an $A_1$ singularity). It follows fom
Table~\ref{table:degres} that $\d(8)=\d^*(8)=1=\d(12)=\d^*(12)$.
Also, Theorem~\ref{theo:springer}(f) implies that
$\det(w_8)=\z_8^{4-d-d_1-d_2-d_3}=\z_8^{-24}=1$ and
$\det(w_{12})=\z_{12}^{4-d-d_1-d_2-d_3}=\z_{12}^{-24}=1$.

\medskip

\item[$\bullet$] {\it Type $G_{30}$.}
Assume here that $W=G_{30}$. Then $d=12$ and
$(d_1,d_2,d_3)=(2,20,30)$. But we have then
$\PM(2,20,30)=\PM(1,10,15)=\PM(1,2,3)$ so $\SC=\{p_2,p_3\}$, so $d_k
\in \{20,30\}$. Note also that $p_2$ is an $A_1$-singularity while
$p_3$ is an $A_2$-singularity. It follows from Table~\ref{table:degres} that  $\d(20)=\d^*(20)=1$ 
$=\d(30)=\d^*(30)$.
Also, Theorem~\ref{theo:springer}(f) implies that
$\det(w_{20})=\z_{20}^{4-d-d_1-d_2-d_3}=\z_{20}^{-60}=1$ and
$\det(w_{30})=\z_{30}^{4-d-d_1-d_2-d_3}=\z_{30}^{-60}=1$.
%
%
\end{itemize}

\medskip

The proof of Lemma~\ref{lem:etale} is complete.
\end{proof}

\subsubsection{The case where $\G=W' \neq W^\slv$}\label{sub:w'}
This case can only occur if $W=G(2,1,4)$ or $G_{28}$.

\medskip

\begin{itemize}
\item[$\bullet$] {\it Type $G(2,1,4)$.}
Assume here that $W=G(2,1,4)$. Then $d=6$ and $f$ is also an
invariant of degree $6$ of $G(2,2,4)$. Since
$W'=G(2,2,4)'=G(2,2,4)^\slv$, the result follows from the previous
subsection.

\medskip

\item[$\bullet$] {\it Type $G_{28}$.}
Then $d \in \{6,8\}$. Then $d_1=2$, $d_3=12$ and $d_2$ is the unique
element of $\{6,8\} \setminus \{d\}$. Then
$$\ZC(f)/W'=\{[x_1:x_2:x_3:j_1:j_2] \in \PM(2,d_2,12,12,12)~|~$$
$$\hphantom{AAAAAAAA}j_1^2=P_{\fb,\O_1}(0,x_1,x_2,x_3)
~\text{and}~j_2^2=P_{\fb,\O_2}(0,x_1,x_2,x_3)\}.$$ The group
$W^\slv/W'$ has order $2$ (we denote by $\s$ its non-trivial
element) and it acts on $\ZC(f)/W'$ as follows:
$$\s([x_1:x_2:x_3:j_1:j_2])=[x_1:x_2:x_3:-j_1:-j_2].$$
So one can check that the ramification locus $\RC$ of
the morphism $\theta : \ZC(f)/W' \longto \ZC(f)/W^\slv$ is defined
by $j_1=j_2=0$ in both cases. We only need to prove that $\RC$ is finite: indeed,
if it is finite, then $\theta$ is unramified in codimension $1$ and
$\ZC(f)/W^\slv$ has only ADE singularities as it was shown
in~\S\ref{sub:slv}, so $\ZC(f)/W'$ has only ADE singularities by
Lemma~\ref{lem:revetement-ade}.

Now,
$$\pi_\fb^{\prime -1}(\RC)=\{p \in \PM(V)~|~j_1(p)=j_2(p)=f(p)=0\}.$$
We only need to prove that $\pi_\fb^{\prime -1}(\RC)$ is finite.
First, let
$$\HC=\{p \in \PM(V)~|~j_1(p)=j_2(p)=0\}.$$
Then the irreducible components of $\HC$ are lines of the form
$\PM(H_1 \cap H_2)$, where $H_1 \in \O_1$ and $H_2 \in \O_2$. This
means that we only need to prove that such a line cannot be entirely
contained in $\ZC(f)$. So, let $H_1 \in \O_1$ and $H_2 \in \O_2$ and
let $s_k$ denote the reflection of $W$ whose reflecting hyperplane
is $H_k$. Let $G=\langle s_1,s_2 \rangle$. Then $V^G=H_1 \cap H_2$
so $\dim V^G = 2$. If $\PM(V^G)$ is entirely contained in $\ZC(f)$,
it then follows from Corollary~\ref{coro:droite contenue} that it is
contained in $\ZC_\sing(f)$: but this contradicts the fact that
$\ZC(f)$ has only ADE singularities.
\end{itemize}

\medskip
The proof of Theorem~\ref{theo:main} is complete.\hfill\qed



%
%
%
%

\renewcommand\thesection{\Alph{section}}
\setcounter{section}{0}

\section*{Appendix A. Surfaces in weighted projective spaces}
\addcontentsline{toc}{section}{Appendix A. Surfaces in weighted
projective spaces} \refstepcounter{section}
\label{app:2-forme}

Let $m \ge 3$ and let $l_0$, $l_1,\dots, l_m$ be positive
integers. We denote by $x_0$, $x_1,\dots, x_m$ the coordinates in
the weighted projective space $\PM(l_0,l_1,\dots,l_m)$ and we fix
$m-2$ polynomials $F_1,\dots, F_{m-2}$ in the variables $x_0$,
$x_1,\dots, x_m$ which are homogeneous of degree $e_1,\dots,
e_{m-2}$ (where $x_k$ is given the degree $l_k$). We consider the
variety
$$X=\{[x_0:x_1:\cdots:x_m] \in \PM(l_0,l_1,\dots,l_m)~|~
\forall~1 \le j \le m-2,~ F_j(x_0,x_1,\dots,x_m)=0\}.$$ Let
$\PM_\smooth$ (resp. $\PM_\sing$) denote the smooth (resp. singular)
locus of $\PM(l_0,l_1,\dots,l_m)$. We assume throughout this section
that the following hold:
\begin{itemize}
\itemth{H1} The weighted projective space $\PM(l_0,l_1,\dots,l_m)$
is well-formed, i.e.
$$\gcd(l_0,\dots,l_{j-1},l_{j+1},\dots,l_m)=1$$
for all $j \in \{0,1,\dots,m\}$.

\itemth{H2} The variety $X$ is a weighted complete intersection,
i.e. $\dim(X)=2$.

\itemth{H3} The variety $X$ is {\it well-formed}, i.e. $\codim_X(X \cap \PM_\sing) \ge 2$.

\itemth{H4} $l_0 + l_1 + \cdots + l_m = e_1 + \cdots + e_{m-2}$.

\itemth{H5} $X$ has only ADE singularities.
\end{itemize}
Note that we do not assume that $X$ is {\it quasi-smooth} (i.e. we
do not assume that the affine cone of $X$ in $\CM^{m+1}$ is smooth
outside the origin~\cite[\S{3.1.5}]{dolga_weighted}). The following
result is certainly well-known but, due to the lack of an
appropriate reference (particularly in the non-quasi-smooth case),
we provide here an explicit proof:

\begin{lemma}\label{lem:canonique}
Under the hypotheses \emph{(H1), (H2), (H3), (H4) and (H5)}, the smooth locus
of the surface $X$ has a non-degenerate $2$-form.
\end{lemma}

\begin{proof}
We set $\PM=\PM(l_0,l_1,\dots,l_m)$ for simplification. Let $U$
denote the smooth locus of $X \cap \PM_\smooth$. By (H3), $X \cap
\PM_\sing$ has codimension $ \ge 2$ in $X$ and so $X \setminus U$
has codimension $\ge 2$ in $X$ by (H5). Again by (H5), it is
sufficient to prove that $U$ admits a non-degenerate $2$-form.

If $0 \le a \le m$ and $1 \le j \le m-2$, we denote $\PM^{(a)}$ the
open subset of $\PM$ defined by $x_a \neq 0$: we identify it with
$\CM^m/\mub_{l_a}$, where the coordinates in $\CM^m$ are denoted by
$(x_0,\dots,x_{a-1},x_{a+1},\dots,x_m)$ and $\mub_{l_a}$ acts
through $\z \cdot
(x_0,\dots,x_{a-1},x_{a+1},\dots,x_m)=(\z^{l_0}x_0,\dots,\z^{l_{a-1}}x_{a-1},
\z^{l_{a+1}}x_{a+1},\dots,\z^{l_m} x_m)$, and we set
$$F_j^{(a)}(x_0,\dots,x_{a-1},x_{a+1},\dots,x_m)=F_j(x_0,\dots,x_{a-1},1,x_{a+1},\dots,x_m).$$
The smooth locus of $\PM^{(a)}$ will be denoted by
$\PM_\smooth^{(a)}$ and, since $\PM$ is well-formed by~(H1), the
above action of $\mub_{l_a}$ on $\CM^m$ contains no reflection and
so $\PM_\smooth^{(a)}$ is the unramified locus of the morphism
$\CM^m \to \CM^m/{\mub_{l_a}}$.

Let $J^{(a)}= \bigl(\partial F_j^{(a)}/\partial x_k\bigr)_{1 \le j
\le m-2,0 \le k \neq a \le m}$ denote the Jacobian matrix of the
family $(F_1^{(a)},\dots,F_{m-2}^{(a)})$. If $b$, $c$ are two
different elements of $\{0,1,\dots,m\}$ which are different from
$a$, we denote by $J_{b,c}^{(a)}$ the $(m-2) \times (m-2)$ minor of
$J^{(a)}$ obtained by removing the two columns numbered by $b$ and
$c$. We set
$$\PM_{\smooth,b,c}^{(a)}=\{p \in \PM_\smooth^{(a)}~|~
J_{b,c}^{(a)}(p) \neq 0\}$$
$$U_{b,c}^{(a)}=X \cap \PM_{\smooth,b,c}^{(a)}.\leqno{\text{and}}$$
By the above description of $\PM_\smooth^{(a)}$ and~(H2), we get
$$U = \bigcup_{\substack{0 \le a,b,c \le m \\ |\{a,b,c\}|=3}}
U_{b,c}^{(a)}.$$

We now define a $2$-form $\omega_{b,c}^{(a)}$ on
$\PM_{\smooth,b,c}^{(a)}$ by \equat \omega_{b,c}^{(a)}= l_a
\frac{dx_b \wedge dx_c}{J_{b,c}^{(a)}}.
\endequat
Let us first explain why this defines a $2$-form on
$\PM_{\smooth,b,c}^{(a)}$. This amounts to show that
$\omega_{b,c}^{(a)}$ is invariant under the action of $\mub_{l_a}$
on the variables $(x_k)_{0 \le k \neq a \le m}$ given by $\xi \cdot
x_k = \xi^{l_k} x_k$. But, if $M$ is a monomial in $J_{b,c}^{(a)}$
of degree $e$ in the variables $(x_k)_{0 \le k \neq a \le m}$, then
$$e \equiv e_1+\cdots + e_{m-2} -
(l_0+l_1+\cdots + l_m - l_a - l_b - l_c) \mod l_a \leqno{(\#)}$$
because the variable $x_a$ is specialized to $1$. So $\xi \in
\mub_{l_a}$ acts on $J_{b,c}^{(a)}$ by multiplication by
$\xi^{l_b+l_c}$ by~(H4). So it acts trivially on
$\omega_{b,c}^{(a)}$.

We now denote by $\omega_{U,b,c}^{(a)}$ the restriction of
$\omega_{b,c}^{(a)}$ to $U_{b,c}^{(a)}$. Note that
$U_{b,c}^{(a)}=U_{c,b}^{(a)}$ but that
$\omega_{U,b,c}^{(a)}=-\omega_{U,c,b}^{(a)}$, so we have to make
some choice. We denote by $\EC$ the set of triples $(a,b,c)$ of
elements of $\{0,1,\dots,m\}$ such that $a < b < c$ or $c < a < b$
or $b < c < a$. Then again
$$U = \bigcup_{(a,b,c) \in \EC} U_{b,c}^{(a)}$$
and we want to show that the family of $2$-forms
$(\omega_{U,b,c}^{(a)})_{(a,b,c) \in \EC}$ glue together to define a
$2$-form on $U$. The argument is standard and will be done in two
steps.

\medskip

\noindent{\it First step: glueing inside an affine chart.} We fix $a
\in \{0,1,\dots,m\}$ and we set $U^{(a)}=U \cap \PM^{(a)}$. Let $b$,
$c$, $b'$, $c' \in \{0,1,\dots,m\}$ be such that $(a,b,c)$,
$(a,b',c') \in \EC$. We need to prove that
$$\left.\omega_{U,b,c}^{(a)}\right|_{U_{b,c}^{(a)} \cap U_{b',c'}^{(a)}}
= \left.\omega_{U,b',c'}^{(a)}\right|_{U_{b,c}^{(a)} \cap
U_{b',c'}^{(a)}}.\leqno{(\sharp)}$$ Proving~$(\sharp)$ is a
computation in $\CM^m$ and amounts to prove that
$$J_{b',c'}^{(a)} dx_b \wedge dx_c =
J_{b,c}^{(a)} dx_{b'} \wedge dx_{c'}\leqno{(\sharp')}$$ on the
variety $\Xhat^{(a)}$ defined by $F_1^{(a)}=\cdots=F_{m-2}^{(a)}=0$
inside $\CM^m$. By applying a power of the cyclic permutation
$(0,1,\dots,m)$ to the coordinates, we may (and we will) assume that
$a=0$ (so that $0 < b < c$ and $0 < b' < c'$). Since $F_j^{(0)}$
vanishes on $\Xhat^{(0)}$, its differential vanishes also on $X$,
which implies that
$$\forall 1 \le j \le m-2,~
\sum_{k=1}^m \frac{\partial F_j^{(0)}}{\partial x_k} dx_k = 0
\quad\text{on $\Xhat^{(0)}$.}$$ Then~$(\sharp')$ is an easy
application of generalized Cramer's rule~\cite{gong}.

\medskip

\noindent{\it Second step: glueing affine charts.} We denote by
$\omega_U^{(a)}$ the glueing of the $2$-forms
$\omega_{U,b,c}^{(a)}$, where $b$, $c$ are such that $(a,b,c) \in
\EC$. Let $a$, $a' \in \{0,1,\dots,m\}$. We need to prove that
$$\omega^{(a)}|_{U^{(a)} \cap U^{(a')}} = \omega^{(a')}|_{U^{(a)} \cap U^{(a')}}.
\leqno{(\flat)}$$ For simplifying the notation, we will assume that
$(a,a')=(0,1)$, the general case being treated similarly. We will
denote by $(x_k)_{0 \le k \neq a \le m}$ the coordinates on
$\PM^{(a)}$ and $(x_k')_{0 \le k \neq a' \le m}$ the coordinates on
$\PM^{(a')}$. Also for simplifying the notation, we will assume that
$U_{1,2}^{(0)} \cap U_{(2,0)}^{(1)} \neq \vide$. So, for
proving~$(\flat)$, we only need to prove that
$$l_0 J_{2,0}^{(1)} dx_1 \wedge dx_2 = l_1 J_{1,2}^{(0)} dx_2' \wedge dx_0'
\qquad\text{on $\PM^{(0)} \cap \PM^{(1)}$.}\leqno{(\flat')}$$ The
variables $(x_k)_{0 \le k \neq a \le m}$ and $(x_k')_{0 \le k \neq
a' \le m}$ are related as follows:
$$
\begin{cases}
x_0'=\DS{\frac{1}{x_1^{l_0/l_1}}},\\
\forall 2 \le k \le m,~x_k'=\DS{\frac{x_k}{x_1^{l_k/l_1}}}.
\end{cases}$$
Therefore $dx_0'=-(l_0/l_1) x_1^{-1-l_0/l_1} dx_1$ and so
$$l_1 dx_2' \wedge dx_0' = l_0 x_1^{-(l_0+l_1+l_2)/l_1}
dx_1 \wedge dx_2.\leqno{(\flat'')}$$ Moreover, since $F_j$ is
homogeneous of degree $e_j$, we get
$$F_j^{(0)}(x_1,x_2,\dots,x_m)=x_1^{e_j/l_1} F_j^{(1)}(x_0',x_2',\dots,x_m').$$
We deduce that
$$\frac{\partial F_j^{(0)}}{\partial x_k}=x_1^{(e_j-l_k)/l_1}\frac{\partial F_j^{(1)}}{\partial x_k'}$$
for all $k \ge 3$ and then
$$J_{1,2}^{(0)}=x_1^{((e_1+\cdots+e_{m-2})-(l_3+\cdots+l_m))/l_1} J_{2,0}^{(1)}.
\leqno{(\flat''')}$$ So $(\flat')$ follows from~$(\flat'')$
and~$(\flat''')$ since $e_1+\cdots + e_{m-2} = l_0+ l_1 + \cdots +
l_m$ by~(H4).
\end{proof}

\begin{cor}\label{coro:canonique}
Under the hypotheses \emph{(H1), (H2), (H3), (H4) and (H5)}, the smooth locus
of the surface $X$ has a trivial canonical sheaf and its minimal
resolution is a smooth K3 surface or an abelian variety.
\end{cor}

\section*{Appendix B. Around ADE singularities}
\addcontentsline{toc}{section}{Appendix B. Around ADE singularities}
\refstepcounter{section}\label{app:auto-ade}

The results of this appendix are certainly well-known. Here, we let
$\Gb\Lb_2(\CM)$ act on the ring of formal power series $\CM[[t, u]]$
naturally by linear changes of the variables. We set $B= \Spec
\CM[[t, u]]$ and we denote by $0$ its unique closed point. We set
$B^\#=B \setminus \{0\}$: it is an open subscheme of $B$. Since $B$
is normal, it follows from Hartog's Lemma that
\equat\label{eq:sections-globales} \OC_{B}(B^\#)=\CM[[t,u]].
\endequat

\begin{lemma}\label{lem:auto-ade}
Let $G$ be a finite subgroup of $\Gb\Lb_2(\CM)$ containing no
reflection and let $\s$ be an automorphism of $\CM[[t, u]]^G$. Then
$\s$ lifts to an automorphism of $\CM[[t, u]]$.
\end{lemma}

\begin{remark}\label{rem:sl2-ref}
If $G$ is a finite subgroup of $\Sb\Lb_2(\CM)$, then $G$ contains no
reflection. This shows that the above lemma applies to ADE
singularities $\CM[[t,u]]^G$.
\end{remark}

\begin{proof}
First, note that $B$ has a trivial fundamental group
by~\cite[expos\'e~I, th\'eor\`eme~6.1]{SGA1}. As it is regular of
dimension 2, its open subset $B^\#$ has also a trivial fundamental
group by \cite[expos\'e~X, corollaire~3.3]{SGA1}. Therefore, the
natural map
$$\pi : B^\# \longto B^\#/G$$
is a universal covering: indeed, the morphism $B \to B/G$ is
ramified only at $0$ because $G$ does not contain any reflection. In
particular, $\pi \circ \s$ is also a universal covering, which means
that $\s$ lifts to an automorphism of $B^\#$ since $B^\#/G$ is
connected. Taking global sections and
using~(\ref{eq:sections-globales}) yields the result.
\end{proof}

\begin{lemma}\label{lem:revetement-ade}
Let $\pi : Y \to X$ be a finite morphism of normal surfaces which is
unramified in codimension $1$. We assume moreover that $X$ has only
ADE singularities. Then $Y$ has only ADE singularities.
\end{lemma}

\begin{proof}
Let $y \in Y$ and let $x=\pi(y)$. Then there exists a finite
subgroup $G$ of $\Sb\Lb_2(\CM)$ such that the completion of
the local ring $\OC_{X,x}$ of $X$ at $x$ is given by $\hat{\OC}_{X,x} \simeq
\CM[[t,u]]^G$. Therefore, the morphism of schemes
$$\pi_y : (\Spec \OCh_{Y,y}) \setminus \{y\} \longto
B^\#/G$$ induced by $\pi$ is unramified by hypothesis, so there
exists a morphism of schemes
$$B^\# \longto (\Spec \OCh_{Y,y}) \setminus \{y\}$$
whose composition with $\pi_y$ is a universal covering of $B^\#/G$
(see the proof of Lemma~\ref{lem:auto-ade}).

Consequently, there exists a subgroup $H$ of $G$ such that
$$(\Spec \OCh_{Y,y}) \setminus \{y\} =
B^\#/H.$$ Taking global sections and applying Hartog's Lemma
together with~(\ref{eq:sections-globales}) yields that
$\OCh_{Y,y}=\CM[[t,u]]^H$.
\end{proof}

Recall that, if $G \subset \Gb\Lb_2(\CM)$, then the only point of
$\CM^2/G$ that might be singular is the $G$-orbit of $0$ (denoted by
$\bar{0}$). The next result is certainly well-known:

\begin{lemma}\label{lem:gl2-ade}
Let $G$ be a finite subgroup of $\Gb\Lb_2(\CM)$ which is generated
by $\Refle(G)$ and let $\G$ be a subgroup of $G$ of index $1$ or
$2$. Then $\bar{0} \in V/\G$ is smooth or an ADE singularity.
\end{lemma}

\begin{proof}
We argue by induction on the order of $G$, the case where $|G|=1$
being trivial. Also, if $\G=G$, then $V/\G$ is smooth so we may
assume that $\G \neq G$. As $\G$ is of index $2$, it is normal and
we denote by $\t : G \to \mub_2$ the unique morphism such that
$\G=\Ker(\t)$. Let $\G_r$ be the subgroup of $\G$ generated by
reflections belonging to $\G$. It is a normal subgroup of $G$ and
$$V/G=(V/\G_r)/(G/\G_r)\qquad\text{and}\qquad V/\G=(V/\G_r)/(\G/\G_r).$$
Two cases may occur:
\begin{itemize}
\item[$\bullet$] If $\G_r \neq 1$, then $V/\G_r$ is isomorphic to a vector space
on which $G/\G_r$ acts linearly as a reflection group since $V/G$ is
smooth (see also~\cite[proposition~3.5]{BBR}). So the result follows
from the induction hypothesis.

\item[$\bullet$] If $\G_r=1$, then $\t(s) \neq 1$ for all $s \in \Refle(G)$.
This implies that all reflections of $G$ have order $2$. Indeed, if
$s \in \Refle(G)$, then $\t(s^2)=1$ so $s^2$ cannot be a reflection,
hence is equal to $1$. This shows in particular that $\t(s)=\det(s)$
for all $s \in \Refle(G)$ and so $\t(w)=\det(w)$ for all $w \in W$.
In particular, $\G \subset \Sb\Lb_2(\CM)$ and the result follows.
\end{itemize}
The proof of the lemma is complete.
\end{proof}

\begin{remark}\label{rem:sl2-optimal}
We explain here why the general result stated in
Lemma~\ref{lem:gl2-ade} is in some sense optimal. Let $d \ge 3$.
Then there exists a reflection group $G$ in $\Gb\Lb_2(\CM)$
admitting a normal subgroup $\G$ such that $W/\G$ is cyclic of order
$d$ and $\CM^2/\G$ admits a non-simple singularity. Take for
instance $G=\mub_d \times \mub_d$ embedded through diagonal
matrices, and $\G \simeq \mub_d$ embedded through scalar
multiplication. Then $\bar{0}$ is not an ADE singularity of
$\CM^2/\G$.

In the same spirit, there exists a reflection group $G$ in
$\Gb\Lb_2(\CM)$ admitting a normal subgroup $\G$ such that $G/\G
\simeq \mub_2 \times \mub_2$ and such that $V/\G$ admits a
non-simple singularity. Take for instance $G=G(4,2,2)$ and
$\G=\Zrm(W)$. Then $G/\G$ is indeed isomorphic to $\mub_2 \times
\mub_2$ and $\G$ is isomorphic to $\mub_4$ acting through scalar
multiplication. So $\bar{0}$ is not an ADE singularity of
$\CM^2/\G$.
\end{remark}

\begin{cor}\label{coro:ade}
Let $X$ be a surface with only ADE singularities and let $G$ be a
group acting on $X$ such that $X/G$ is smooth. Let $\G$ be a
subgroup of $G$ of index $2$. Then $X/\G$ has only ADE
singularities.
\end{cor}

\begin{proof}
Let $x \in X$. It is sufficient to show that $X/\G_x$ has an ADE
singularity at the image of $x$. Note that we know that $X/G_x$ is
smooth at the image of $x$. Also, $\G_x$ has index $1$ or $2$ in
$G_x$. This shows that we may, and we will, assume that $G=G_x$ (and
so $\G=\G_x$).

By hypothesis, there exists a subgroup $H$ of $\Sb\Lb_2(\CM)$ such
that the complete local ring $\hat{\OC}_{X,x}$ is isomorphic to
$\CM[[t,u]]^H$. Let us identify $\OCh_{X,x}$ with $\CM[[t,u]]^H$, so
that the group $G$ acts on $\CM[[t,u]]^H$. By
Lemma~\ref{lem:auto-ade}, the action of an element $g \in G$ on
$\CM[[t,u]]^H$ lifts to an automorphism $\gti$ of $\CM[[t,u]]$: we
fix such a $\gti$ for all $g \in G$. Note that $\{\gti h~|~h \in
H\}$ is the set of all lifts of $g$ to $\CM[[t,u]]$. Let
$$\Gti=\{\gti h ~|~g \in G~\text{and}~h \in H\}.$$
Then $\Gti$ is a group (as $\gti h \gti' h'$ is a lift of $gg'$ so
belongs to $\Gti$) and we have an exact sequence
$$1 \longto H \longto \Gti \longto G \longto 1.$$
Let $\Gamt$ denote the inverse image of $\G$ in $\Gti$. Note that
$(\CM[[t,u]]^H)^{G}=\CM[[t,u]]^\Gti$ and
$(\CM[[t,u]]^H)^{\G}=\CM[[t,u]]^\Gamt$.

Now, by hypothesis, $(\CM[[t,u]]^H)^{G}$ is regular. This shows that
$\Gti$ acts as a reflection group on the tangent space of $\Spec
\CM[[t,u]]$ at its unique closed point, and so the result follows
from Lemma~\ref{lem:gl2-ade} because $\Gamt$ has index $1$ or $2$ in
$\Gti$.
\end{proof}

\end{document}